\documentclass{amsart}
\usepackage[margin=1.2in]{geometry}
\usepackage{amssymb}
\usepackage{amscd}
\usepackage[all]{xy}
\usepackage{bbm}
\usepackage{mathrsfs}
\usepackage{enumerate}
\usepackage{tikz}
\usepackage{stmaryrd}
\usepackage{etoolbox}
\usepackage{array}
\usepackage{verbatim}
\usepackage{hyperref}
\usepackage{stmaryrd}
\usepackage{bbding}

\newcommand{\R}{\mathbb{R}}

\newcommand{\eps}{\varepsilon}

\newcommand{\tensor}{\otimes}

\newcommand{\del}{\nabla}

\newcommand{\bd}{\partial}
\newcommand{\cl}{\overline}

\newcommand{\la}{\langle}
\newcommand{\ra}{\rangle}

\renewcommand{\div}{\operatorname{div}}

\newcommand{\grad}{\del}

\newcommand{\an}{\operatorname{An}}

\theoremstyle{plain}
\newtheorem{theorem}{Theorem}

\newtheorem{prop}[theorem]{Proposition}
\newtheorem{lemma}[theorem]{Lemma}

\newtheorem{question}[theorem]{Question}

\theoremstyle{definition}

\newtheorem{rem}[theorem]{Remark}

\author{Liam Mazurowski}
\address{Department of Mathematics, Southern Connecticut State University, New Haven, CT 06515}
\email{mazurowskil1@southernct.edu}

\author{Xuan Yao}
\address{Department of Mathematics, University of Chicago, Chicago, IL 60637}
\email{dorisxyao@uchicago.edu}

\title{Scalar Curvature Lower Bounds and Convergence in Measure}

\begin{document}

\begin{abstract}
    Assume that smooth metrics $g_k$ converge in measure to a smooth metric $g$. Gromov asked the following question: if all the metrics $g_k$ have non-negative scalar curvature does it follow that $g$ also has non-negative scalar curvature? We answer Gromov's question in dimension three. We show that, without further assumptions, the metric $g$ need not have non-negative scalar curvature and in fact $g$ may be completely arbitrary. However, if one assumes in addition that the identity maps $(M,g_k)\to (M,g)$ are uniformly bi-Lipschitz, then indeed $g$ must have non-negative scalar curvature.  Our method of proof can also be used to show that, under an almost Euclidean entropy condition, scalar curvature lower bounds will persist under convergence of the metrics together with their inverses in $L^p$ for suitably large $p$.  
\end{abstract}

\maketitle

\section{Introduction}

Gromov proved that scalar curvature lower bounds are preserved under the $C^0$-convergence of smooth metrics to a smooth limit \cite{gromov2014dirac,gromov2019four}.  Later, Bamler \cite{bamler2016ricci} gave an alternative proof using Ricci flow \cite{bamler2016ricci}.  In dimension three, there are also proofs using harmonic functions \cite{mazurowski2026quantification} and $\mu$-bubbles \cite{mazurowski2026scalar}. It is a very interesting problem to understand whether and to what extent the $C^0$-convergence assumption can be weakened in this result.  In this direction, recent work of the authors \cite{mazurowski2026scalar} showed that $C^0$-convergence can be relaxed to volume above distance below type convergence. On the other hand, uniform convergence of distance functions \cite{lee2024metric} and intrinsic flat convergence \cite{krandel2024smooth} do not suffice. 

Many other relaxations of $C^0$-convergence are possible.  For example, in Four Lectures on Scalar Curvature \cite[Section 3.1.3]{gromov2019four}, Gromov posed the following question about scalar curvature and convergence in measure. 

\begin{question}[Gromov]
\label{Gromov-Question}
    Assume that smooth metrics $g_k$ converge in measure to a smooth limit $g$. If each metric $g_k$ has non-negative scalar curvature, does it follow that $g$ has non-negative scalar curvature as well? 
\end{question}

\begin{rem}
    Gromov remarks that this is most likely to be true if one assumes in addition that the identity maps $(M,g_k)\to (M,g)$ are uniformly bi-Lipschitz. 
\end{rem}

In this paper, we answer Gromov's question about convergence in measure in dimension three. Our first result shows that the answer to Question \ref{Gromov-Question} is no in general. In fact, convergence in measure alone puts no restrictions at all on the limiting metric. 

\begin{theorem} Fix any smooth metric $g$ on the unit ball $B^3$. Then there are smooth metrics $g_k$ on $B^3$ such that $R(g_k) \ge 1$ for all $k$ and $g_k$ converges to $g$ in measure.
\end{theorem}

Our second result shows that the answer to Question \ref{Gromov-Question} is yes under the extra assumption suggested by Gromov. 

\begin{theorem}
\label{main-theorem}
    Assume $M^3$ is a closed three manifold. Assume that smooth metrics $g_k$ on $M$ converge in measure to a smooth metric $g$. Further assume that there is a constant $K > 1$ such that the identity maps $(M,g_k)\to (M,g)$ are all $K$-bi-Lipschitz. Then if each metric $g_k$ has non-negative scalar curvature so does $g$. 
\end{theorem}

\begin{rem}
    Previous work of the authors \cite{mazurowski2026quantification} showed that there is a universal constant $\eps_0 > 0$ so that the above result is true whenever $K \le 1+\eps_0$. In higher dimensions, Lee \cite{lee2026quantification} proved the same result when $K \le 1 + \eps_0(n)$ for a constant $\eps_0(n) > 0$ depending only on the dimension. A key point of Theorem \ref{main-theorem} is that $K$ may be arbitrarily large.   
\end{rem}

The proof of Theorem \ref{main-theorem} uses $p$-harmonic functions with $p > 2$ to analyze the scalar curvature.  Essentially the same method can also be used to study scalar curvature lower bounds under non-collapsed convergence of the metric together with its inverse in $L^P$. Here, following \cite{lee2023dp}, we use Perelman's entropy and Perelman's $\nu$-functional \cite{perelman2002entropy} to measure non-collapsing.  In this setting, we can prove the following result. 

\begin{theorem}
\label{main-lp}
There are constants $P > 3$ and $\delta > 0$ so that the following holds.   Let $g_k$ be a sequence of smooth metrics on a closed manifold $M^3$ satisfying the uniform non-collapsing condition $\nu(g_k,2)\ge -\delta$.  Assume there is another smooth metric $g$ such that $g_k \to g$ and $g_k^{-1}\to g^{-1}$ in $L^P$. Then if each metric $g_k$ has non-negative scalar curvature so does $g$. 
\end{theorem} 

Note that if $g_k \to g$ in measure and $g_k$ is uniformly bi-Lipschitz to $g$ via the identity, then $g_k \to g$ in $L^p$ and $g_k^{-1}\to g^{-1}$ in $L^p$ for every finite $p$.  However, while the bi-Lipschitz assumption does give a uniform lower bound $\nu(g_k,2)\ge -C$, the constant $C$ need not be small.  Thus neither Theorem \ref{main-theorem} nor Theorem \ref{main-lp} implies the other.

\begin{rem} Theorem \ref{main-lp} is closely connected with the $d_P$-theory of Lee-Naber-Neumayer \cite{lee2023dp}. Indeed, 
for any sequence of metrics $g_k$ satisfying $R(g_k)\ge 0$ and $\nu(g_k,2)\ge -\delta$, the $d_P$-compactness theory \cite[Theorem 8.1 and Proposition 8.2]{lee2023dp} implies that (after passing to a subsequence and replacing $g_k$ by $\psi_k^*g_k$ for some diffeomorphisms $\psi_k$) the metrics $g_k$ will converge to a limit $g$ in $L^P$. In fact, by \cite[Remark 8.3]{lee2023dp} the inverse metrics $g_k^{-1}$ will also converge to $g^{-1}$ in $L^P$. Of course, this limiting metric $g$ will not be smooth in general but will only belong to $L^P$. 
\end{rem}

\subsection{Discussion} 

One of Gromov's many influential ideas asserts that scalar curvature type phenomena persist well-below the $C^2$ regularity threshold needed to define scalar curvature in the classical sense.  This is well-demonstrated by Gromov's $C^0$-convergence theorem \cite{gromov2014dirac,gromov2019four} which shows that a smooth metric which is a $C^0$-limit of smooth metrics with non-negative scalar curvature itself must have non-negative scalar curvature. The further development of this perspective remains a very active area of research.  For example, it is an important problem to determine the extent to which the $C^0$-convergence assumption can be weakened in this result.  

\begin{question}
\label{general-question} 
Assume that smooth manifolds $(M_k,g_k)$ with non-negative scalar curvature converge (weakly) to a smooth limit $(M,g)$. Which notions of weak convergence ensure that the limit space also has non-negative scalar curvature? 
\end{question}

Recent work of the authors \cite{mazurowski2026scalar} shows, for example, that $C^0$-convergence can be relaxed to volume above distance below type convergence; this answers questions of Allen \cite[Question 6.3]{allen2024oberwolfach}, Gromov \cite[Page 333]{gromov2019four}, and Sormani \cite[Question 4.2 and Remark 4.5]{sormani2023conjectures}.  On the other hand, it is known that uniform convergence of distance functions \cite{lee2024metric} and intrinisic flat convergence \cite{krandel2024smooth} are not enough. In this paper we contribute to the study of this question by showing that scalar curvature lower bounds are preserved when smooth metrics converge to a smooth limit in measure with uniform bi-Lipschitz control. 

We note that, given a sequence of smooth manifolds with non-negative scalar curvature, it is often possible (with addition assumptions) to extract a weak limiting space.  For example, one may appeal to the intrinsic flat compactness theorem of Wenger \cite{sormani2011intrinsic,wenger2011compactness} or the $d_p$-convergence theory of Lee-Naber-Neumayer \cite{lee2023dp}.  Of course, this limiting space may not be smooth in general.  Hence one would also like to ask Question \ref{general-question} when the limiting space is no longer smooth. An important first step in this direction is to develop suitable synthetic notions of non-negative scalar curvature on non-smooth spaces. 

Many different notions have been proposed.  For example, Burkhardt-Guim \cite{burkhardt2019pointwise} defined a notion of non-negative scalar curvature for $C^0$-metrics using Ricci flow, building on earlier work of Bamler \cite{bamler2016ricci}.  Lee and LeFloch \cite{lee2015positive} defined non-negative scalar curvature in the sense of distributions for $C^0\cap W^{1,n}$ metrics and obtained a positive mass theorem in this class. Another possibility is to say that a continuous metric has non-negative scalar curvature if it can be uniformly approximated by smooth metrics with almost non-negative scalar curvature.  The authors \cite{mazurowski2026positive} recently proved a positive mass theorem with rigidity for continuous metrics with non-negative scalar curvature in this sense. We mention that Burkhardt-Guim's work \cite{burkhardt2019pointwise} shows that manifolds with non-negative scalar curvature in the Ricci flow sense have non-negative scalar curvature in the sense of approximations. Likewise, the very recent work of Lee-Litzinger-Simon  \cite{lee2026spaces} shows that $C^0\cap W^{1,n}$ metrics with non-negative scalar curvature in the sense of distributions also have non-negative scalar curvature in the sense of approximations. Finally, we note that Gromov's dihedral rigidity theorem for polyhedra \cite{brendle2024scalar,li2020polyhedron,li2024dihedral,wang2021gromov} can also be used to give a synthetic definition of non-negative scalar curvature. 

Sormani \cite{sormani2023conjectures} conjectured that a sequence of smooth manifolds with non-negative scalar curvature, uniform volume and diameter bounds, and a uniform minA lower bound (a type of non-collapsing condition) should converge in the volume preserving intrinsic flat sense to a limit space with non-negative scalar curvature.  There are many partial results on this conjecture; see for example \cite{allen2026scalar,park2019compactness,sormani2025geometric,tian2024compactness,wang2026scalar}. On the other hand, the drawstring examples of Kazaras and Xu \cite{kazaras2023drawstrings}  suggest that one must be careful about the meaning of non-negative scalar curvature on the limiting space.  In another direction, Lee-Naber-Neumayer \cite{lee2023dp} studied sequences of manifolds with (almost) non-negative scalar curvature and nearly Euclidean entropy (another type of non-collapsing condition). They showed that, after passing to a subsequence, one can always obtain convergence in the $d_p$-sense to a limiting space with nice $W^{1,p}$ analysis.   

This further demonstrate the need to probe the precise threshold at which non-negative scalar curvature is preserved by limits.  To this end, one can ask whether the uniform bi-Lipschitz assumption can be weakened in Theorem \ref{main-theorem} and whether the entropy condition can be weakened in Theorem \ref{main-lp}.  This motivates the following question.

\begin{question}
Fix any number $1 \le p < \infty$. Assume that smooth metrics $g_k$ and $g$ satisfy $g_k \to g$ in $L^p$ and $g_k^{-1}\to g_k^{-1}$ in $L^p$. If each $g_k$ has non-negative scalar curvature, does it follow that $g$ has non-negative scalar curvature? If this is not true, is there any non-collapsing condition weaker than $\nu(g_k,2)\ge -\delta$ under which it is true? 
\end{question}

\subsection{Sketch of Proof} In \cite{mazurowski2026quantification}, the authors developed a method for using harmonic functions to study the scalar curvature of metrics which are controlled only in $C^0$. In the first part of this paper, we extend this method to the case of $p$-harmonic functions.  As we will explain, the use of $p$-harmonic functions with $p > 2$ plays an essential role in the proof of Theorem \ref{main-theorem}. 

Fix some $1 < p < 3$ and consider a positive $p$-harmonic function $u$ with connected level sets on a manifold $(M^3,g)$. Let 
\[
a = \frac{3-p}{p-1}
\]
so that $\vert x\vert^{-a}$ is $p$-harmonic on $\R^3$. We define a bulk integral quantity $D_p(r)$ which captures information about $u$ and $\grad u$ on the annular-type region $\{(4r)^{-a} < u < r^{-a}\}$. This reduces to the quantity $D(r)$ employed in \cite{mazurowski2026positive} and \cite{mazurowski2026quantification}  when $p=2$. It is straightforward to show that the quantity $D_p(r)$ depends continuously on $u$ in $W^{1,q}$ for every $q > 3$. 

Now assume that $g$ has non-negative scalar curvature. If $u$ has no critical points, the quantity $D_p(r)$ can be re-expressed as an integral of the quantity $F_p(t)$ introduced by Agostiniani-Mantegazza-Mazzieri-Oronzio in \cite{agostiniani2022riemannian}. Using the monotonicity formula for $F_p(t)$ from \cite{agostiniani2022riemannian}, it then follows that $r\mapsto r^a D_p(r)$ is non-decreasing. In general, the function $u$ may have critical points, and in fact it is  possible a priori that the set of regular values of $u$ is quite sparse. Since the function $F_p(t)$ is only defined for regular values, the above argument breaks down. Nevertheless, in this case one can still use the approximate monotonicity formula for the regularized functions $F_p^\eps(t)$ in \cite{agostiniani2022riemannian} and then pass to a limit as $\eps \to 0$ to deduce that $r\mapsto r^a D_p(r)$ is non-decreasing. We remark that the function $D_p(r)$ has the nice property that it is manifestly well-defined for all values of $r$ even in the presence of critical points, unlike $F_p(t)$ which may only be defined for a very small subset of the possible $t$ values. 

We now sketch the proof of Theorem \ref{main-theorem}.  The proof is by contradiction.  Suppose that the metrics $g_k$ all have non-negative scalar curvature but that there is a point $y$ with $R_g(y) < 0$. Fix a number $2 < p < 3$ to be specified later.   Consider the rescaled metrics $g_\rho(x) = g(\rho x)$ and $g_{k,\rho}(x) = g_k(\rho x)$ defined in the $g$-exponential coordinates centered at $y$.  Let $\an = B(0,16) - B(0,1/2)$ and let $u$ be the $p$-harmonic function in the $g_\rho$ metric which is equal to $\vert x\vert^{-a}$ on the boundary of $\an$. Let $D_p(r)$ be the above function associated to $u$. Using the fact that $g_\rho \to g_{\text{euc}}$ as $\rho \to 0$ and  exploiting the rotational symmetry of the Euclidean metric, the annulus $\an$, and the function $\vert x\vert^{-a}$, we can show that the sign of $D_p(1)-2^a D_p(2)$ has the opposite sign of $R_g(y)$ for small $\rho$. Since $R_g(y) < 0$, it follows that $D_p(1) > 2^aD_p(2)$ for sufficiently small $\rho$. This argument is designed to avoid having to work directly with a Green's function for the $p$-Laplacian, which was not  well-understood prior to the very recent paper \cite{agostiniani2026estimates}.

We fix such a small $\rho$ and relabel $g_\rho$ as $g$ and $g_{k,\rho}$ as $g_k$ for convenience. Let $u_k$ be the $p$-harmonic function on $\an$ in the $g_k$ metric which is equal to $\vert x\vert^{-a}$ on the boundary of $\an$. Since $g_k$ has non-negative scalar curvature and $u_k$ has connected level sets by the maximum principle, the above monotonicity discussion implies that $D_{k,p}(1) \le 2^a D_{k,p}(2)$ for all $k$.  Hence to obtain a contradiction, it suffices to show that $D_{k,p}(1) \to D_p(1)$ and $D_{k,p}(2) \to D_{k,p}(2)$ as $k\to \infty$. Using the fact that $g_k\to g$ in measure and the uniform bi-Lipschitz property, it is straightforward to show that $u_k \to u$ strongly in $W^{1,p}$. However, as explained above, to conclude the convergence of $D_{k,p}$ to $D_p$, one actually needs convergence in $W^{1,q}$ for some $q > 3$. 

To obtain this, we use Gehring's lemma. It is well-known that Gehring's lemma implies that the gradient of a $p$-harmonic function actually belongs to $L^{p+\eps}$ for some $\eps > 0$. In fact, by keeping careful track of the constants and using the $K$-bi-Lipschitz bound and the fact that $p$ belongs to the compact range $[2,3]$, one can show that the improvement of integrability constant $\eps = \eps(K) > 0$ depends only on $K$ and not $p$ or $k$. Thus we can select $p = p(K)$ close enough to $3$ so that $p + \eps(K) > 3$, and we obtain that the gradients $\grad u_k$ are uniformly bounded in $L^{p+\eps(K)}$. Since $u_k \to u$ in $W^{1,p}$ and $p + \eps(K)> 3$, it then follows easily that $u_k \to u$ in $W^{1,q}$ for some $q > 3$. This completes the proof.   

The proof of Theorem \ref{main-lp} follows essentially the same strategy.  The main difference is that we use the $d_P$-theory \cite{lee2023dp} to obtain the convergence of the $p$-harmonic functions and the uniform improvement of integrability.

\section{Examples}

In this section, we construct examples to show that scalar curvature lower bounds are not preserved under convergence in measure without further assumptions. Note that the metric 
\[
g_{0} = \left(\frac{4}{4+\vert x\vert^2}\right)^2 g_{\text{euc}}
\]
is a round metric of sectional curvature 1. The main observation that allows one to construct such examples is simply that 
\[
\| g_{0}-g_{\text{euc}}\|_{C^0(B(0,r))} \to 0
\]
as $r\to 0$.

We will also need the connect sum construction of Gromov and Lawson \cite{gromov1980classification}.  Fix two round spheres $S^3_1$ and $S^3_2$ of radius 1.  Fix small balls $B(p_1,s_1) \subset S^3_1$ and $B(p_2,s_2)\subset S^3_2$.  By the connect sum construction, there is a metric $h$ on the connect sum $S^3_1 \# S^3_2$ with scalar curvature $R(h) \ge 6-\eps$ such that $h$ agrees with the round metric on $S^3_1$ outside $B(p_1,s_1)$ and $h$ agrees with the round metric on $S^3_2$ outside $B(p_2,s_2)$. See \cite[Proposition 2.1.2]{sweeney2024tunnels} for the details of this construction in the case of positive scalar curvature lower bounds.  

We start with a special case that illustrates the main idea. 

\begin{prop}
    There are smooth metrics $g_k$ on the unit ball $B^3$ such that $R(g_k)\ge 1$ for all $k$ but $g_k \to g_{\operatorname{euc}}$ in measure. 
\end{prop}

\begin{proof}
It suffices to show that for each $\eps > 0$ there is a smooth metric $g_\eps$ on $B^3$ with $R(g_\eps) \ge 1$ which satisfies 
\[
\vert \{x\in B^3: \vert g_\eps(x) - g_{\text{euc}}(x)\vert \ge \eps \}\vert \le \eps,
\]
where the vertical bars denote the Euclidean measure of a set. So fix some $\eps > 0$.  Select $r_0 > 0$ small enough that 
\[
\|g_0 - g_{\text{euc}}\|_{C^0(B(0,r_0))} < \eps. 
\]
Choose a finite collection of balls $B_i = B(x_i,r_i)$ for $i=1,\hdots,k$ such that 
\begin{itemize}
    \item[(i)] $\overline B_i\subset B^3$ for all $i$,
    \item[(ii)] the closed balls $\overline B_i$ are all disjoint,
    \item[(iii)] the set $B^3 - \left(\cup_{i=1}^k B_i\right)$ has Euclidean measure less than $\eps$, and
    \item[(iv)] $r_i \le r_0$ for all $i$.
\end{itemize}
Now select $a > 1$ close enough to 1 that the balls $B_i' = B(x_i,a r_i)$ still satisfy 
\begin{itemize}
    \item [(i')] $\overline {B_i'} \subset B^3$, and
    \item[(ii')] the closed balls $\overline{B_i'}$ are all disjoint. 
\end{itemize}
Define the set $C = B^3 - \left(\cup_{i=1}^k B_i'\right)$. 

We now proceed to define $g_\eps$. Fix some $i\in \{1,\hdots,k\}$. By the connect sum construction, there exists a metric $g_{\eps,i}$ on $B^3$ such that $R(g_{\eps,i}) \ge 6-\eps$, and $g_{\eps,i} = g_0$ outside $B_i'$, and the restriction of $g_{\eps,i}$ to $B_i'$ represents a round sphere of radius 1 with a small ball deleted and a neck attached connecting back to the round metric on $B^3 - B_i'$. We can further suppose that the restriction of $g_{\eps,i}$ to $B_i$ is given by the formula
\[
g_{\eps,i} = \left(\frac{4}{4+\vert x-x_i\vert^2}\right)^2 g_{\text{euc}}.
\]
Thus $B_i$ corresponds to a small ball in the sphere of radius 1 that was attached, and the remainder of this sphere together with the neck and gluing regions are all contained in the small annulus $B_i' - B_i$. 
\[
\includegraphics[width=4in]{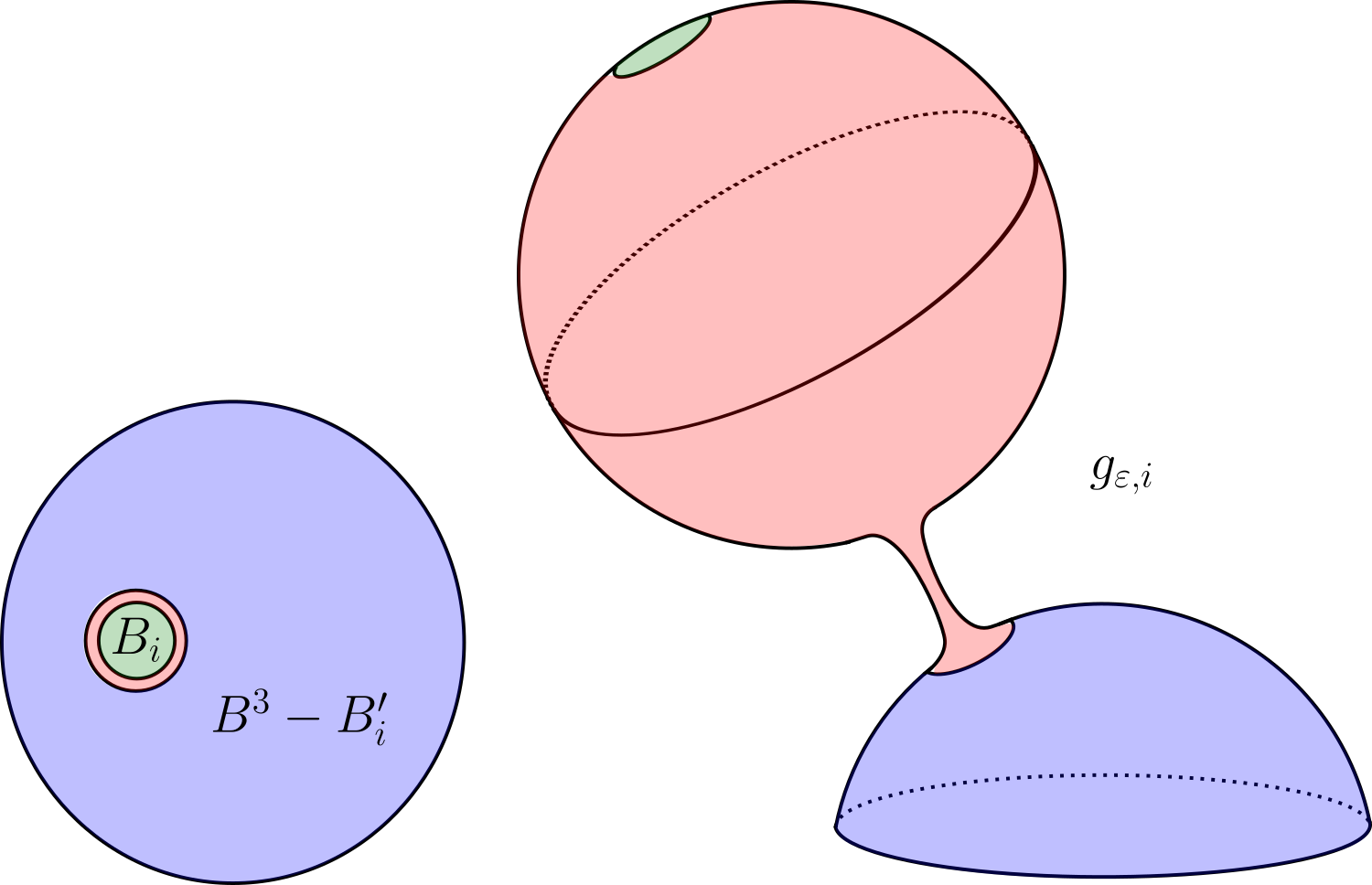}
\]
Finally, we define $g_\eps$ so that $g_\eps = g_{\eps,i}$ on $B_i'$ and $g_\eps = g_0$ on $C$. It is easy to verify that $g_\eps$ is as required. Indeed, the metric $g_\eps$ has $R(g_\eps)\ge 6-\eps$, and satisfies $\vert g_\eps(x) - g_{\text{euc}}(x)\vert < \eps$ for all $x\in \cup_{i=1}^k B_i$, and the Euclidean measure of $B^3 - \cup_{i=1}^k B_i$ is less than $\eps$. 
\end{proof}

By making minor modifications in the above construction, one has the following more general result. 

\begin{theorem}
    Fix any smooth metric $g$ on the unit ball $B^3$. Then there are smooth metrics $g_k$ on $B^3$ such that $R(g_k) \ge 1$ for all $k$ and $g_k$ converges to $g$ in measure.
\end{theorem}

\begin{proof}
Let $g$ be a smooth metric on $B^3$.  For any symmetric, positive definite, $3\times 3$ matrix $A$,  the metric 
\[
(g_A)(x) = \left(\frac{4}{4 + x^T A x}\right)^2 A
\]
is a round metric of sectional curvature 1. Since $g$ is smooth, we have 
\[
\left\|\left(\frac{4}{4 + (x-x_0)^T g(x_0) (x-x_0)}\right)^2 g(x_0)  - g(x)\right\|_{C^0(B(x_0,r))} \to 0
\]
as $r\to 0$ uniformly in $x_0$. Hence we can simply repeat the above construction, but make the metric $g_\eps$ equal to 
\[
\left(\frac{4}{4 + (x-x_i)^T g(x_i) (x-x_i)}\right)^2 g(x_i)
\]
on $B_i$. 
\end{proof}

\section{The Annular Quantity}
\label{Section-Monotonicity}

Fix some $1 < p < 3$. Define 
\[
a  := \frac{3-p}{p-1}
\]
and take $\vert x\vert^{-a}$ as the model $p$-harmonic function on Euclidean space.  
Consider the annulus $\an = B(0,16)- B(0,1/2)$ equipped with a Riemannian metric $g$. Let $u$ be a $p$-harmonic function on $\an$ in the $g$ metric which equals $\vert x\vert^{-a}$ on the boundary of $\an$. 

\subsection{Preliminaries}  We are going to use a function defined by Agostiniani-Mantegazza-Mazzieri-Oronzio \cite{agostiniani2022riemannian}.  However, we will use somewhat different notation as explained below.  To clarify, we will put a tilde over the objects from the original paper \cite{agostiniani2022riemannian} whenever there is some chance of confusion.  

Consider the $p$-harmonic function $u$ defined above. We set $H = -\div({\grad u}/{\vert \grad u\vert})$.  Define $\tilde u = 1 - u$ and let $\widetilde H = \div(\grad \tilde u/\vert \grad \tilde u\vert)= H$. Let 
\[
c_p = \left(\frac{\operatorname{Cap}_p(\tilde u)}{4\pi}\right)^{\frac{1}{p-1}}
\]
where $\operatorname{Cap}_p(\tilde u)$ is given by 
\[
\operatorname{Cap}_p(\tilde u) = \int_{\{\tilde u= s\}} \vert \grad \tilde u\vert^{p-1}\, d\sigma
\]
for any regular value $s$ of $\tilde u$. Further define 
\[
t_p = \left(\frac{ c_p}{a}\right)^{1/a}, \quad \alpha_p(t) = 1 - \left(\frac{t_p}{t}\right)^a.
\]
In \cite{agostiniani2022riemannian}, they define   
\begin{align*}
\widetilde F(t) &= 4\pi t - \frac{t^{\frac 2{p-1}}}{ c_p} \int_{\{\tilde u = \alpha_p(t)\}} \widetilde H \vert \grad \tilde u\vert\, d\sigma + \frac{t^{\frac{5-p}{p-1}}}{ c_p^2}\int_{\{\tilde u = \alpha_p(t)\}} \vert \grad \tilde u\vert^2\, d\sigma\\
&= 4\pi t -  c_p^{-1} t^{1+a} \int_{\{\tilde u = \alpha_p(t)\}} \widetilde H \vert \grad \tilde u\vert\, d\sigma +  c_p^{-2} t^{1+2a} \int_{\{\tilde u = \alpha_p(t)\}} \vert \grad \tilde u\vert^2\, d\sigma
\end{align*}
for regular values $\alpha_p(t)$. 

Now observe that $u = t^{-a} \iff \tilde u = 1-t^{-a} \iff \tilde u = \alpha_p(t_p t)$. We define the re-parameterized function 
\begin{align*}
F(t) &= t_p^{-1} \widetilde F(t_p t) \\
&= 4\pi t - a^{-1} t^{1+a} \int_{\{u = t^{-a}\}} H\vert \grad u\vert\, d\sigma + a^{-2} t^{1+2a} \int_{\{u = t^{-a}\}} \vert \grad u\vert^2\, da.
\end{align*} 
We will write things in terms of this quantity $F$ associated to $u$. If $g$ has non-negative scalar curvature, then the monotonicity result from \cite{agostiniani2022riemannian} implies that $F(t_1) \le F(t_2)$ whenever $t_1 < t_2$ are such that $t_1^{-a}$ and $t_2^{-a}$ are regular values of $u$. 

\subsection{The Smooth Case} 
Now suppose for the moment that $u$ is smooth with no critical points. Note that in this case all level sets of $u$ are spheres. Moreover,  we have the following differentiability relation: 
\[
\frac{d}{dt} \int_{\{u=t^{-a}\}} \vert \grad u\vert^2\, da = -a^2 t^{-a-1} \int_{\{u=t^{-a}\}} H\vert \grad u\vert\, da. 
\]
Thus we can compute 
\begin{align*}
E(r,s) &:= \int_{r+s}^{2r+2s} \frac{F(t)}{t^{a+2}}\, dt \\
&= \frac{4\pi(1-2^{-a})}{a(r+s)^{-b}} + a^{-3} \int_{\{u=(2r+2s)^{-a}\}} \frac{\vert \grad u\vert^2}{u}\, d\sigma - a^{-3}\int_{\{u=(r+s)^{-a}\}} \frac{\vert \grad u\vert^2}{u}\, d\sigma. 
\end{align*}
Now let $\psi$ be a smooth bump function supported in $(0,1)$. Define the function  
\[
D(r) = \int_0^r \psi(s/r) E(r,s)\, ds. 
\]
We can calculate 
\begin{align*}
D(r) &= r^{1-a} c_\psi + \frac 1{a^4} \int \left[\frac 1 2 \psi\left(\frac{1}{2r u^{1/a}}-1\right)-\psi\left(\frac{1}{ru^{1/a}}-1\right)\right] \frac{\vert \grad u\vert^3}{u^{2+1/a}}\, dv\\
&= r^{1-a} c_\psi + \int \varphi(r,u) \vert \grad u\vert^3 \, dv
\end{align*}
where $c_\psi$ is a constant depending only on $\psi$ and we set 
\begin{equation}
\varphi(r,t) =  \frac{1}{a^4 t^{2+1/a}} \left[\frac 1 2 \psi\left(\frac{1}{2r t^{1/a}}-1\right)-\psi\left(\frac{1}{rt^{1/a}}-1\right)\right].
\label{phi-equation}
\end{equation}
It is easy to check that if $F$ is non-decreasing then  $r\mapsto r^a D(r)$ is also non-decreasing.  

Indeed, by definition of $D(r)$ and the smoothness assumption on the $p$-harmonic function $u$, we have 
\[
r^aD(r)=\int_0^{r}\psi(s/r)r^a\left[\int_{r+s}^{2r+2s}\frac{F(t)}{t^{a+2}}dt\right]ds.
\]
We make change of variables $m^{1/a}(r+s)=t$ in the inner integral and obtain
\begin{align*}
r^aD(r)=\int_0^r\psi(s/r)r^a\left[\int_1^{2^a}\frac{F((r+s)m^{1/a})}{m^{2+1/a}(r+s)^{a+1}}dm\right]ds.
\end{align*}
We further substitute $v=s/r$ and obtain
\begin{align*}
    r^aD(r)=\int_0^1\frac{\psi(v)}{(1+v)^{a+1}}\left[\int_{1}^{2^a}\frac{F((1+v)rm^{1/a})}{m^{2+1/a}}dm\right]dv.
\end{align*}
The monotonicity of $r^aD(r)$ then follows from the monotonicity of $F(t)$.
\subsection{The General Case} 

In general, the function $u$ may not be smooth, and in this case the function $F$ may not be defined for many values of $t$.  The function $D(r)$ defined above has the advantage that it is manifestly well-defined for all values of $r$.  Assume that $g$ has non-negative scalar curvature. We would like to show that $r\mapsto r^a D(r)$ is still monotone even when the argument in the previous section breaks down.

As in \cite{agostiniani2022riemannian}, we first fix a large annular type region $M_{S,T}=\{S\leq u\leq T\}$, where $S,T$ are some regular values of $u$. We consider the $\eps$-perturbed $p$-harmonic function satisfying 
\begin{align*}
    \left\{\begin{matrix}
        \div\left(\left(\sqrt{|\nabla u_{\eps}|^2+\eps^2}\right)^{p-2}\nabla u_{\eps}\right)=0 & \text{ in } M_{S,T}\\
        u_{\eps}=S & \text{ on } \{u=S\}\\
        u_{\eps}=T & \text{ on } \{u=T\}.
    \end{matrix}\right.
\end{align*}
We summarize the basic analytical properties of $u_{\eps}$ in the following lemma.

\begin{lemma}\label{lem: appro reg}
    The function $u_{\eps}$ converges to $u$ on any compact subset of $M_{S,T}$ in $C^{1,\beta}$-topology as $\eps\to 0$. Furthermore, $u_{\eps}$ converges to $u$ smoothly on any compact set of $M_{S,T}$ with critical points removed. The critical points of $u_{\eps}$ have measure zero, and $u_{\eps}$ is smooth up to the boundary.
\end{lemma}

\begin{proof}
    See \cite{ dibenedetto1982c,benedetio1983interior} for the convergence results. Since the $\eps$-perturbed $p$-harmonic function is a divergence form non-degenerate equation, we can apply the weak and strong maximum principle as well as the Hopf lemma and the unique continuation property. As a result, we obtain the desired regularity properties.
\end{proof}

The strategy is to apply arguments similar to those in the previous section to $u_{\eps}$ and obtain the desired monotonicity by taking $\eps\to 0$.
Let
\[
D_{\eps}(r):=\int_0^r\psi(s/r)\left[\int_{r+s}^{2r+2s}\frac{F_{\eps}(t)}{t^{a+2}}\,dt\right]\, ds,
\]
where
\[
F_{\eps}(t):=4\pi t-a^{-1}t^{1+a}\int_{\{u_{\eps}=t^{-a}\}}H|\nabla u_{\eps}|\, da+a^{-2}t^{1+2a}\int_{\{u_{\eps}=t^{-a}\}}|\nabla u_{\eps}|^2\, da.
\]
As before, our notion of $F_{\eps}(t)$ is a reparameterization of $\tilde{F}_{\eps}(t)$, which is defined as 
\[
\tilde{F}_{\eps}(t)=4\pi t-c_{p,\eps}^{-1}t^{1+a}\int_{\{\tilde{u}_{\eps}=\alpha_{p}^{\eps}(t)\}}\widetilde{H}|\nabla \tilde{u}_{\eps}|\, da+a^{-2}t^{1+2a}\int_{\{\tilde{u}_{\eps}=\alpha_p^{\eps}(t)\}}|\nabla \tilde{u}_{\eps}|^2\, da
\]
in \cite{agostiniani2022riemannian}. More precisely, we have $F_\eps(t) = t_{p,\eps}^{-1} \tilde F_\eps(t_{p,\eps} t)$. In the above, 
the capacity of the $\eps$-perturbed function is defined as 
\[
\operatorname{Cap}_p(\tilde{u}_{\eps}):=\int_{\{\tilde{u}_{\eps}=s\}}\left(\sqrt{|\nabla \tilde{u}_\eps|^2+\eps^2}\right)^{p-2}|\nabla\tilde{u}_{\eps}|\, da,
\]
and 
\[
c_{p,\eps}=\left(\frac{\operatorname{Cap}_p(\tilde{u}_{\eps})}{4\pi}\right)^{\frac{1}{p-1}},\quad t_{p,\eps}=\left(\frac{c_{p,\eps}}{a}\right)^{1/a}, \quad \alpha_{p}^{\eps}(t)=1-\left(\frac{t_p}{t}\right)^a.
\]
We recall the almost monotonicity formula from \cite{agostiniani2022riemannian}, adapted to our notation. 
\begin{lemma}[Approximate monotonicity \cite{agostiniani2022riemannian}]
      For any $S<t_1<t_2<T$ regular values of $u_{\eps}$, we have that 
      \[
      F_{\eps}(t_2)-F_{\eps}(t_1)\geq -C\eps\int_{\{t_{2}^{-a}\leq u_{\eps}\leq t_{1}^{-a}\}}\frac{|\nabla u_{\eps}|^2}{u_{\eps}^{3+1/a}}\, dv,
      \]
      where $C$ is a positive constant independent of $t_1,t_2$.
\end{lemma}

Now, it is straightforward to prove that $r^aD_{\eps}(r)$ is almost monotone. For fixed $r_2>r_1>0$, we have that 
\begin{align*}
    r_2^aD_{\eps}(r_2)-r_1^aD_{\eps}(r_1)=&\int_0^1\frac{\psi(v)}{(1+v)^{a+1}}\left[\int_1^{2^a}\frac{F_{\eps}((1+v)r_2m^{1/a})-F_{\eps}((1+v)r_1m^{1/a})}{m^{2+1/a}}\, dm\right]\, dv\\
    \geq &\int_0^1\frac{\psi(v)}{(1+v)^{a+1}}\left\{\int_1^{2^a}\frac{-C\eps}{m^{2+1/a}}\left[\int_{\{(4r_2)^{-a}\leq u_{\eps}\leq r_1^{-a}\}}\frac{|\nabla u_{\eps}|^2}{u_{\eps}^{3+1/a}}\, d\mu\right]dm\right\}\,dv\\
    \geq &-C(r_1,r_2,\psi)\eps \int_{M_{S,T}}|\nabla u_{\eps}|^2\, d\mu,
\end{align*}
where $C(r_1,r_2,\psi)$ is a positive constant depending only on $r_1,r_2,\psi$. 

\begin{lemma}
    For a fixed $r>0$, we have that 
    \[
    \lim_{\eps\to 0}D_{\eps}(r)=D(r).
    \]
\end{lemma}
\begin{proof}

Since $u_\eps$ is smooth and $\{\grad u_\eps = 0\}$ has measure zero, the vector field $\vert \grad u_\eps\vert \grad u_\eps$ is $C^1$ with 
\[
\div(\vert \grad u_\eps\vert \grad u_\eps) = \langle \nabla |\nabla u_{\eps}|, \nabla u_{\eps}\rangle +|\nabla u_{\eps}|\Delta u_{\eps}
\]
almost everywhere. 
Thus for $t_1>t_2>0$, regular values of $u_{\eps}$, we can apply the divergence theorem to obtain 
\begin{align*}
    \int_{\{u_{\eps}=t_1^{-a}\}}|\nabla u_{\eps}|^2\, da_g&-\int_{\{u_\eps=t_2^{-a}\}}|\nabla u_{\eps}|^2\, da_g\\
    &=\int_{\{t_1^{-a}\leq u_{\eps}\leq t_2^{-a}\}}\div(|\nabla u_{\eps}|\nabla u_{\eps})\, d\mu\\
    &=\int_{\{t_1^{-a}\leq u_{\eps}\leq t_2^{-a}\}}\left(\langle \nabla |\nabla u_{\eps}|, \nabla u_{\eps}\rangle +|\nabla u_{\eps}|\Delta u_{\eps}\right)\, d\mu\\
    &=\int_{\{t_1^{-a}\leq u_{\eps}\leq t_2^{-a}\}}\left(1-\frac{(p-2)|\nabla u_{\eps}|^2}{|\nabla u_{\eps}|^2+\eps^2}\right)\langle \nabla|\nabla u_{\eps}|,\nabla u_{\eps}\rangle\, d\mu \\
    &=\int_{\{t_1^{-a}\leq u_{\eps}\leq t_2^{-a}\}}-aH|\nabla u_{\eps}|^2+\frac{2(p-2)\eps^2\langle \nabla|\nabla u_{\eps}|,\nabla u_{\eps}\rangle}{(p-1)\left(|\nabla u_{\eps}|^2+\eps^2\right)}\, d\mu.
\end{align*}
In the last equality, we used that 
\[
H=\frac{\Delta u_{\eps}}{|\nabla u_\eps|}-\frac{\langle \nabla|\nabla u_{\eps}|,\nabla u_{\eps}\rangle}{|\nabla u_{\eps}|^2}=-(p-2)\frac{\langle \nabla|\nabla u_{\eps}|,\nabla u_{\eps}\rangle}{|\nabla u_{\eps}|^2+\eps^2}-\frac{\langle \nabla |\nabla u_{\eps}|,\nabla u_{\eps}\rangle}{|\nabla u_{\eps}|^2}
\]
away from the measure zero set of critical points of $u_\eps$. 

For almost every $t>0$, we let
\[
P(t):=-at^{-1-a}\int_{\{u_{\eps}=t^{-a}\}}\left(1-\frac{(p-2)|\nabla u_{\eps}|^2}{|\nabla u_{\eps}|^2+\eps^2}\right)\langle \nabla |\nabla u_{\eps}|,\frac{\nabla u_{\eps}}{|\nabla u_{\eps}|}\rangle\, d\mu.
\]
From the previous calculations, the co-area formula, and the fact that $\{\nabla u_{\eps}=0\}$ has zero measure, for regular values $t_1>t_2>0$, we have
\begin{align*}
\int_{\{u_{\eps}=t_1^{-a}\}}|\nabla u_{\eps}|^2\, da_g-\int_{\{u_{\eps}=t_2^{-a}\}}|\nabla u_{\eps}|^2\, da_g=\int_{t_2}^{t_1}P(t)\, dt.
\end{align*}
Moreover, note that for any fixed  regular values $t_1<t_2$, we have that 
\[
\int_{t_2}^{t_1}|P(t)|\, dt<C\int_{\{t_1^{-a}\leq u_{\eps}<t_2^{-a}\}}|\nabla^2u_{\eps}|\, d\mu<\infty.
\]
This implies that 
\[
t\to\int_{\{u_{\eps}=t^{-a}\}}|\nabla u_{\eps}|^2\, da
\]
has an absolutely continuous representative. Therefore, we can regard this as absolutely continuous function of $t$.

Now, it is straightforward to compute
\begin{align*}
    \frac{d}{dt}\int_{\{u_{\eps}=t^{-a}\}}&|\nabla u_{\eps}|^2\, da=\\
    &-a^2t^{-a-1}\int_{\{u_{\eps}=t^{-a}\}}H|\nabla u_{\eps}|\, da+\frac{2(p-2)}{p-1}\int_{\{u_{\eps}=t^{-a}\}}\frac{\eps^2\langle \nabla|\nabla u_{\eps}|,\nabla u_{\eps}\rangle}{|\nabla u_{\eps}|\left(|\nabla u_{\eps}|^2+\eps^2\right)}\, da
\end{align*}
for almost every $t$. 
Applying integration by parts as before, we have
\begin{align*}
    D_{\eps}(r)=&r^{1-a}c_{\psi}+\int\varphi(r,u_{\eps})|\nabla u_{\eps}|^3\, d\mu+\int_0^r\psi(s/r)Q(r,s)\, ds,
\end{align*}
where 
\begin{align*}
Q(r,s)=&\frac{2(p-2)\eps^2}{(p-1)}\int_{r+s}^{2r+2s}a^{-2}t^{a-1}\int_{\{t^{-a}\leq u_{\eps}\leq (r+s)^{-a}\}}\frac{\langle \nabla|\nabla u_{\eps}|,\nabla u_{\eps}\rangle}{|\nabla u_{\eps}|^2+\eps^2}\, d\mu\, dt\\
    &-a^{-3}(2r+2s)^a\frac{2(p-2)\eps^2}{(p-1)}\int_{\{(2r+2s)^{-a}\leq u_{\eps}\leq (r+s)^{-a}\}}\frac{\langle \nabla |\nabla u_{\eps}|,\nabla u_{\eps}\rangle}{|\nabla u_{\eps}|^2+\eps^2}\, d\mu.
\end{align*}
The  elementary inequality
\[
\frac{\eps|\nabla u_{\eps}|}{|\nabla u_{\eps}|^2+\eps^2}\leq \frac{1}{2}
\]
implies that
\begin{align*}
\int_0^r\psi(s/r)Q(r,s)\, ds&\leq C(r)\left\vert\int_0^r\psi(s/r)\int_{\{(2r+2s)^{-a}\leq u_{\eps}\leq (r+s)^{-a}\}}\frac{2(p-2)\eps^2\langle \nabla |\nabla u_{\eps}|,\nabla u_{\eps}\rangle}{(p-1)(|\nabla u_{\eps}|^2+\eps^2)}\right\vert\\
&\leq \eps C(r)|M_{S,T}|\max_{M_{S,T}}|\nabla^2 u_{\eps}|\to 0,\quad \text{ as } \eps\to 0.
\end{align*}
The proof is now complete with the result of Lemma \ref{lem: appro reg}.
\end{proof}

Combined with the above results, we have now shown that $r\to r^a D(r)$ is non-decreasing.

\section{Expansion on Small Annuli}
\label{Section-Expansion}

Fix a number $2 < p < 3$. Define $\an = B(0,16) - B(0,1/2)$.  Fix some large, positive integer $k$.  Let $\Gamma^{k,\alpha}(\an)$ be the set of $C^{k,\alpha}$ Riemannian metrics on $\an$.  
We define a map 
\[
\mathcal U\colon \Gamma^{k,\alpha}(\an) \to C^{k,\alpha}(\an) 
\]
by letting $\mathcal U(g)$ be the unique $p$-harmonic function in the $g$ metric which is equal to $u_0(x) = \vert x\vert^{-a}$ on the boundary of $\an$. 

Next, let $\an(1) = B(0,4)-B(0,1) \subset \an$.  Define a functional $\mathcal D_1$ from a neighborhood of $(g_{\text{euc}},u_0)$ in $\Gamma^{k,\alpha}(\an)\times C^{k,\alpha}(\an)$ to $\R$ by 
\begin{gather*}
\mathcal D_1(g,u) = c_\psi + \int_{\an(1)} \varphi(1,u) \vert \grad^g u\vert^3\, dv_g,
\end{gather*}
where $\varphi$ is defined by \eqref{phi-equation}. 
Then define a functional $\mathcal Q_1$ from a neighborhood of $g_{\text{euc}}$ in $\Gamma^{k,\alpha}(\an)$ to $\R$ by setting $\mathcal Q_1(g) = \mathcal D_1(g,\mathcal U(g))$.

Now consider the following geometric situation. Fix a smooth metric $g$ on $M^3$ and fix a point $q\in M$. We introduce geodesic normal coordinates $x$ near $q$, and then rescale them to get metrics $g_r(x) = g(rx)$ on $\an$.  By the Taylor expansion of the metric in geodesic normal coordinates, we have 
\[
g_r = \delta_{ij} -\frac{r^2}{3} R_{ikj\ell}(q) x^k x^\ell + O(r^3),
\]
where $O(r^3)$ denotes a quantity whose $C^{k,\alpha}$ norm on $\an$ is at most $C r^3$. 
We would like to understand the behavior of $\mathcal Q_1(g_r)$ as $r\to 0$. 

\begin{prop}
We have $\mathcal Q_1(g_r) = b_1 r^2  R(q) + O(r^3)$ as $r\to 0$. Here $b_1 > 0$ is a fixed constant.
\end{prop}

\begin{proof}
    First, we want to prove that 
   $
   g\mapsto \mathcal U(g)
   $
   is a smooth functional on a small neighborhood of $g_{euc}$.
   Let $\mathcal E=C^{k,\alpha}(\overline{\an};\text{Sym}^2 T^*\an)$, $X=\{v\in C^{k,\alpha}(\bar{\an}): v|_{\partial\an}=0\}$, and $Y=C^{k-2,\alpha}(\overline{\an})$. Consider the functional
   \[
   \Phi: \mathcal O\subset \mathcal E\times X\to Y,\quad \Phi(h,v):=\Delta_{p,g_h}(u_0+v),
   \]
   where $\mathcal O$ is a sufficiently small neighborhood of $(0,0)$ and $g_h=g_{\text{euc}}+h$.

   Since $|\nabla ^{g_{\text{euc}}}u_0|>c_0$ for any $x\in\an$, after shrinking $\mathcal O$ small enough, we have that 
   \[
   |\nabla^{g_h}(u_0+v)|>\frac{c_0}{2},\quad \forall (h,v)\in\mathcal O.
   \]
   Therefore, from the explicit formula 
   \[
\Phi(h,v)=\frac{1}{\sqrt{\det g_h}}\partial_i\left(\sqrt{\det g_h}\left(g_h^{ml}\partial_mu_v\partial_l u_v\right)^{\frac{p-2}{2}}g_h^{ij}\partial_j u_v\right)
   \]
   we know $\Phi: \mathcal O\to Y$ is a smooth functional, where $u_v:=u_0+v$. 
   It is straightforward to compute that 
   \[
   Tu:=D_v\Phi(0,0)[w]=\div\left(|\nabla u_0|^{p-2}\left(I+\frac{\nabla u_0}{|\nabla u_0|}\tensor \frac{\nabla u_0}{|\nabla u_0|}\right)\nabla w\right),
   \]
   where all the gradients are with respect to the Euclidean metric.
   Since $|\nabla u_0|>c_0$, we know that $T$ is a uniformly elliptic operator and hence $T: X\to Y$ is an isomorphism. By the implicit function theorem, it then follows that 
   \[
   g\mapsto \mathcal U(g)
   \]
   is indeed a smooth functional in a small neighborhood of the Euclidean metric.

   Since $|\nabla^{g_\text{euc}} u_0|>c_0$ for any $x\in \an$, it is straightforward to check that $\mathcal D_1$ is a smooth functional in a small neighborhood of $(g_{\text{euc}},u_0)$. Therefore, we see that $\mathcal Q_1(g):=\mathcal D_1(g,\mathcal U(g))$ is a smooth functional in a small neighborhood of $g_{\text{euc}}$.
   Therefore, we have a local expansion of $\mathcal Q_1$ near $g_{\text{euc}}$ of the form 
   \[
   \mathcal Q_1(g_h)=\mathcal Q_1(g_{\text{euc}})+L(h)+O(\|h\|_{C^{k,\alpha}}^{{2}}).
   \]
Recall that the metric $g_r$ has the local expansion of 
   \[
   g_r=\delta_{ij}-\frac{r^2}{3}R_{ijkl}(q)x^kx^l+O(r^3),
   \]
   and that $\mathcal Q_1(g_{\text{euc}})=0$. 
   Thus, by the rotational symmetry of the problem, we have that 
   \[
   \mathcal Q_1(g_r)=b_1r^2R(q)+O(r^3),
   \]
    for some fixed constant $b_1$ that does not depend on $g$.  Indeed, every linear map from the space of algebraic curvature tensors to $\R$ which is $\operatorname{SO}(3)$ invariant must be a multiple of the scalar curvature; see for example \cite[Section 1G]{besse1987einstein}.

   Next, we need to verify that $b_1$ is indeed a positive constant. Consider the round metric $g^*$ on the unit sphere $S^3$. It has the form
   \[
   g^*=d\rho^2+\sin^2\rho\, dS^2,
   \]
   where $dS^2$ is the standard metric on the unit sphere $S^2$. 
   The rescaled metrics for $g^*$ are given by
   \[
   g^*_r=d\rho^2+\frac{\sin^2 r\rho}{r^2}\, dS^2,
   \]
   which is exactly the formula for the rotational symmetric metric $g_{\kappa}$ with sectional curvature $\kappa=r^2$. In the Cartesian normal coordinates, we have the expression 
   \[
   (g_{\kappa})_{ij}=\delta_{ij}-\frac{\kappa}{3}(|x|^2\delta_{ij}-x_ix_j)+O(\kappa^2).
   \]
   By the expansion of $\mathcal Q_1$ we computed earlier, we have that 
   \[
   \mathcal Q_1(g_{\kappa})=6b_1 \kappa+O({\kappa^{3/2}}).
   \]
   It then follows that 
   \[
   6b_1=\frac{d}{d\kappa}\mathcal Q_1(g_{\kappa})\vert_{\kappa=0}.
   \]
   Let $u_{\kappa}=\mathcal U(g_{\kappa})$. Since $g_{\kappa}$ is rotationally symmetric, $u_{\kappa}$ must be a radial function, and the $p$-harmonic equation reduces to the ODE
   \[
   \frac{d}{d\rho}\left[s_{\kappa}(\rho)^2|u_{\kappa}'|^{p-2}u_{\kappa}'\right]=0,
   \]
   where $s_{\kappa}$ is the solution to the ODE
   \[
   s_{\kappa}''+\kappa s_{\kappa}=0, \quad s_{\kappa}(0)=0, \quad s_{\kappa}'(0)=1.
   \]

   Note that if $u_{\kappa}'$ achieves zero, then it must be identically zero, which is a contradiction. Therefore, we conclude that $u_{\kappa}'$ must have a sign. Since $u_{\kappa}(\frac{1}{2})>u_{\kappa}(16)$ by the boundary condition, we know that 
   \[
   u_{\kappa}'<0,\quad \forall x\in \an.
   \]
One then obtains the explicit formula of $u_{\kappa}$ via direct integration. This yields
    \[
    u_{\kappa}(\rho)=2^a-C(\kappa)\int_{\frac{1}{2}}^\rho s_{\kappa}^{\frac{-2}{p-1}}\, d\rho,
    \]
    where $C(\kappa)$ is positive function of $\kappa$ solving the following equation
    \[
    16^{-a}=2^a-C(\kappa)\int_{\frac{1}{2}}^{16} s_{\kappa}(\rho)^{\frac{-2}{p-1}}\, d\rho.
    \]
    The Taylor expansion of sine  near $0$ gives 
    \[
    s_{\kappa}(x)=\rho-\frac{\kappa}{6}\rho^3+O(\kappa^2),
    \]
    and we also have the expansion  
    \[
    C(\kappa)=a-\frac{\kappa a}{3}\mu+O(\kappa^2),
    \]
    where 
    \[
    \mu=\frac{a(a+1)}{2(2-a)}\cdot\frac{2^{4(2-a)}-2^{a-2}}{2^a-2^{-4a}}>0
    \]
    is a constant depending on $a$. 

    By definition, we have that 
    \[
    \mathcal Q_1(g_{\kappa})=c_{\psi}+4\pi C(\kappa)^3\int_{1}^4\mathcal M(u_{\kappa}(\rho))s_{\kappa}(\rho)^{2-3\beta}\, d\rho,
    \]
    where 
    \[
    \mathcal M(t)=\varphi(1,t),\quad \beta=\frac{2}{p-1},
    \]
    and we used that 
    \[
    dv_{g_{\kappa}}=s_{\kappa}(\rho)^2d\rho \, d\omega,
    \]
    where $d\omega$ is the standard area form of unit $S^2$.

    By  Taylor expansion, we have
    \[
    u_{\kappa}(\rho)=\rho^{-a}+v(\rho)\kappa+O(\kappa^2),
    \]
    and 
    \[
    \mathcal M(u_{\kappa})=\mathcal M(\rho^{-a})+\mathcal M'(\rho^{-a})v(\rho)\kappa+O(\kappa^2),
    \]
    where
    \[
    v(\rho)=\frac{2^a-\rho^{-a}}{3}\mu-\frac{a(a+1)}{6(2-a)}\cdot(\rho^{2-a}-2^{a-2}).
    \]
    Plugging in the expansion into the expression of $\mathcal Q_1(g_{\kappa})$, we have 
    \begin{align*}
        \mathcal Q_1(g_{\kappa})=&c_{\psi}+4\pi\int_1^4a^3\rho^{2-3\beta}\left[\mathcal M(\rho^{-a})+\left(\frac{3\beta-2}{6}\rho^2\mathcal M(\rho^{-a})+\mathcal M'(\rho^{-a})v(\rho)\right)\kappa\right]\, d\rho\\
        &-4\pi\mu\int_1^4a^3\rho^{2-3\beta}\mathcal M(\rho^{-a})\kappa\, d\rho+     O(\kappa^2).
    \end{align*}
Therefore, we have 
    \begin{align*}
    \frac{d}{d\kappa}\mathcal Q_1(g_{\kappa})\vert_{\kappa=0}=&4\pi\int_1^{4}a^3\rho^{2-3\beta}\left(\frac{3\beta-2}{6}\rho^2\mathcal M(\rho^{-a})+\mathcal M'(\rho^{-a})v(\rho)\right)\, d\rho\\
    &-4\pi\mu\int_1^4a^3\rho^{2-3\beta}\mathcal M(\rho^{-a})\, d\rho\\
    =&4\pi a^3\int_1^4\frac{3a+1}{6}\rho^{1-3a}\mathcal M(\rho^{-a})+\frac{d}{d\rho}\left(-\frac{1}{a}\mathcal M(\rho^{-a})\right)\rho^{-2a}v(\rho)\, d\rho\\
    &-4\pi\mu a^3\int_1^4\rho^{-1-3a}\mathcal M(\rho^{-a})\, d\rho 
    \\
    =&4\pi a^3\int_1^4\frac{3a+1}{6}\rho^{1-3a}\mathcal M(\rho^{-a})+\frac{1}{a}\mathcal M(\rho^{-a})\frac{d}{d\rho}\left(\rho^{-2a}v(\rho)\right)\\
    &-4\pi\mu a^3\int_1^4\rho^{-1-3a}\mathcal M(\rho^{-a})\, d\rho. 
    \end{align*}
    The last inequality follows from integration by parts and the fact that
    $
    \mathcal M(4^{-a})=\mathcal M(1^{-a})=0.
    $

    Next, we directly compute the term 
\begin{align*}
    \frac{1}{a}\mathcal M(\rho^{-a})\frac{d}{d\rho}\left(\rho^{-2a}v(\rho)\right)=&\frac{1}{a}\mathcal M(\rho^{-a})\left(-2a\rho^{-2a-1}v(\rho)+\rho^{-2a}v'(\rho)\right).
\end{align*}
Since we have 
\[
v(\rho)=\frac{2^a-\rho^{-a}}{3}\mu-\frac{a(a+1)}{6(2-a)}(\rho^{2-a}-2^{a-2}),
\]
direct computation gives
\[
v'(\rho)=\frac{a}{3}\mu\rho^{-a-1}-\frac{a(a+1)}{6}\rho^{1-a}.
\]
Plugging in the above expressions, we have 
\begin{align*}
    \frac{1}{a}\mathcal M(\rho^{-a})\frac{d}{d\rho}\left(\rho^{-2a}v(\rho)\right)=&\frac{1}{a}\mathcal M(\rho^{-a})\left(-2a\rho^{-2a-1}v(\rho)+\rho^{-2a}v'(\rho)\right)\\
    =&\frac{1}{a}\mathcal M(\rho^{-a})\left(\frac{2a}{3}\mu \rho^{-3a-1}+\frac{a^2(a+1)}{6(2-a)}\rho^{1-3a}+L(a)\rho^{-2a-1}\right)\\
    &+\frac{1}{a}\mathcal M(\rho^{-a})\left(\frac{a}{3}\mu\rho^{-3a-1}-\frac{a(a+1)}{6}\rho^{1-3a}\right)\\
    =&\mathcal M(\rho^{-a})\mu\rho^{-3a-1}+\frac{a^2-1}{3(2-a)}\mathcal M(\rho^{-a})\rho^{1-3a}+L(a)\mathcal M(\rho^{-a})\rho^{-2a-1},
\end{align*}
where $L(a)$ is a constant depending only on the value of $a$.

Now note that 
\begin{align*}
        \int_1^4\rho^{-2a-1}\mathcal M(\rho^{-a})\, d\rho=&\int_1^4 \frac{1}{a^4}\left[\frac{1}{2}\psi(\frac{\rho}{2}-1)-\psi(\rho-1)\right]\, d\rho\\
        =&\frac{1}{a^4}\int_1^2\psi(\rho-1)\, d\rho-\frac{1}{a^4}\int_1^4\psi(\rho-1)\, d\rho = 0.
    \end{align*}
Thus we have 
\begin{align*}
    \int_1^4\frac{1}{a}\mathcal M(\rho^{-a})\frac{d}{d\rho}(\rho^{-2a}v(\rho))\, d\rho=\int_1^4\mathcal M(\rho^{-a})\mu\rho^{-3a-1}+\frac{a^2-1}{3(2-a)}\mathcal M(\rho^{-a})\rho^{1-3a}\, d\rho.
\end{align*}
It follows that   \begin{align*}
    \frac{d}{d\kappa}\mathcal Q_1(g_{\kappa})\vert_{\kappa=0}=&4\pi a^3\int_1^4\frac{5a-a^2}{6(2-a)}\mathcal M(\rho^{-a})\rho^{1-3a}\, d\rho 
    \\
    =&\frac{2\pi(5-a)}{3(2-a)}\int_1^4\rho^{2-a}\left[\frac{1}{2}\psi(\frac{\rho}{2}-1)-\psi(\rho-1)\right]\, d\rho\\
    =&\frac{2\pi(5-a)}{3(2-a)}\int_1^2 2^{2-a}\rho^{2-a}\psi(\rho-1)\, d\rho-\frac{2\pi(5-a)}{3(2-a)}\int_1^4\rho^{2-a}\psi(\rho-1)\, d\rho.
    \end{align*}  
   Finally, since $\frac{5-a}{2-a}>0$ and $\psi$ is a positive bump function supported on $(0,1)$, we obtain that $\frac{d}{d\kappa}\mathcal Q_1(g_{\kappa})\vert_{\kappa=0}$ is positive and  the proof is complete.
   \end{proof}

Likewise, let $\an(2) = B(0,8)-B(0,2)$ and define 
\[
\mathcal D_2(g,u) = 2^{1-a}c_\psi + \int_{\an(2)} \varphi(2,u) \vert \grad^g u\vert^3\, dv_g
\]
and $\mathcal Q_2(g) = \ 2^a \mathcal D_2(g,\mathcal U(g))$. 

\begin{prop}
    We have $\mathcal Q_2(g_r) = b_2 r^2 R(q) + O(r^3)$ as $r \to 0$ for a constant $b_2 > b_1$. 
\end{prop}

\begin{proof}
    Similar arguments show that 
    \[
    \mathcal Q_2(g_r)=b_2r^2R(q)+O(r^3).
    \]
    We only need to compute the constant $b_2$ and compare it with $b_1$.
    Plugging in the same test metric family $g_{\kappa}$, we have  
    \[
    6b_2=\frac{d}{d\kappa}\mathcal Q_2(g_{\kappa})\vert_{\kappa=0}.
    \]
    By definition of $\mathcal Q_2(g_{\kappa})$, we have
    \begin{align*}
    &\mathcal Q_2(g_{\kappa})\\
    =&2^{1-a}c_{\psi}+4\pi C(\kappa)^3\int_2^8 2^{2a+1}\mathcal M(2^au_{\kappa}(\rho))s_{\kappa}(\rho)^{2-3\beta}\, d\rho\\
    =&2^{1-a}c_{\psi}+4\pi a^3\int_2^82^{-a}(\frac{\rho}{2})^{2-3\beta}\left[\mathcal M((\frac{\rho}{2})^{-a})+\left(\frac{3\beta-2}{6}\rho^2\mathcal M((\frac{\rho}{2})^{-a})+\mathcal M'((\frac{\rho}{2})^{-a})2^av(\rho)\right)\kappa\right]\, d\rho\\
    &-4\pi a^3\mu \int_2^8 2^{-a}(\frac{\rho}{2})^{2-3\beta}\mathcal M((\frac{\rho}{2})^{-a})\kappa\, d\rho+ O(\kappa^2).
    \end{align*}
The same arguments as above show that 
\begin{align*}
    \frac{d}{d\kappa}\mathcal Q_2(g_{\kappa})\vert_{\kappa=0}=&4\pi a^3\int_2^82^{-a}(\frac{\rho}{2})^{2-3\beta}\left(\frac{3\beta-2}{6}\rho^2\mathcal M((\frac{\rho}{2})^{-a})+\mathcal M'((\frac{\rho}{2})^{-a})2^av(\rho)\right)\, d\rho\\
    &-4\pi a^3\mu\int_2^82^{-a}(\frac{\rho}{2})^{2-3\beta}\mathcal M((\frac{\rho}{2})^{-a})\, d\rho\\
    =&4\pi a^3\int_2^8 2^{-a}(\frac{\rho}{2})^{-3a+1}\frac{3a+1}{6}\mathcal M((\frac{\rho}{2})^{-a})+\frac{d}{d\rho}\left(-\frac{1}{a}\mathcal M((\frac{\rho}{2})^{-a})\right)2^{2a+1}\rho^{-2a}v(\rho)\, d\rho\\
    &-4\pi a^3\mu\int_2^8 2^{-a}(\frac{\rho}{2})^{-3a-1}\mathcal M((\frac{\rho}{2})^{-a})\, d\rho\\
    =&4\pi a^3\int_2^8 \frac{a}{2-a}2^{-a}(\frac{\rho}{2})^{-3a-1}\mathcal M((\frac{\rho}{2})^{-a})\, d\rho\\
    =&2^{1-a}\cdot 4\pi a^3\int_1^4\frac{a}{2-a}\rho^{-3a-1}\mathcal M(\rho^{-a})\, d\rho.
\end{align*}
Therefore, we conclude 
\[
b_2=2^{1-a}b_1>b_1,
\]
since 
\[
1-a=2\cdot\frac{p-2}{p-1}>0.
\]
This completes the proof. 
\end{proof}

\section{Improvement of Integrability}

Fix $2 \le p \le 3$ and define $a = \frac{3-p}{p-1}$.  Let $g$ be a smooth metric on 
\[
\an = B(0,16) - B(0,1/2)
\]
 which is $K$-bi-Lipschitz to the Euclidean metric.  Assume that $u$ is a $p$-harmonic function on $\an$ in the $g$-metric.  We want to show there is a uniform improvement of integrability for the gradient.  This result is well-known, but we include the proof for the reader's convenience since it is important to keep careful track of the dependence of certain constants. 

The fact that $g$ is $K$-bi-Lipschitz to the Euclidean metric implies that 
\[
K^{-2}\vert \xi\vert^2 \le g^{ij} \xi_i\xi_j \le K^2 \vert \xi\vert^2
\]
for all vectors $\xi \in \R^3$. 
The $p$-harmonic equation implies that 
\[
\int_{\an} (g^{k\ell} u_k u_\ell)^{\frac{p-2}{2}} g^{ij} u_i \eta_j \sqrt{\det g} \, dx = 0
\]
for any smooth function $\eta$ compactly supported in $\an$.  We re-write this as 
\[
\int_{\an} A(x,\grad u) \cdot \grad \eta \, dx = 0 
\]
where $A$ is given by 
\[
A(x,\xi) =  \sqrt{\det g} (g^{k\ell} \xi_k \xi_\ell)^{\frac{p-2}{2}} g^{ij} \xi_i \bd_j
\]
and the dot product and gradient denote the Euclidean dot product and gradient.  The fact that $g$ is $K$-bi-Lipschitz to the Euclidean metric  implies that 
\begin{gather*}
\vert A(x,\xi)\vert \le K^{3+p} \vert \xi \vert^{p-1} \le K^6 \vert \xi\vert^{p-1},\\
A(x,\xi)\cdot \xi \ge  K^{-3-p} \vert \xi \vert^{p}\ge K^{-6} \vert \xi\vert^p.
\end{gather*}
We now proceed to derive the improvement of integrability.

The first step is to prove a Caccioppoli type inequality.  Let $Q_r$ denote a family of concentric cubes with side length $r$.

\begin{prop}
For any concentric cubes $Q_r\subset Q_{2r}\subset\subset \an$ there is an estimate 
\[
\int_{Q_r} \vert \grad u\vert^p\, dx \le \frac{C}{r^p} \int_{Q_{2r}} \vert u- u_{Q_{2r}}\vert^p\, dx
\]
where $u_{Q_{2r}}$ is the average of $u$ on $Q_{2r}$ and $C = C(K)$ depends only on $K$. 
\end{prop} 

\begin{proof}
Let $\varphi$ be a smooth function with $\varphi \equiv 1$ on $Q_r$, and $\varphi\equiv 0$ outside $Q_{2 r}$, and $\vert \grad \varphi\vert \le C/r$. We test the equation against $\eta = \varphi^p (u-u_{Q_{2 r}})$ to get  
\begin{align*} 
0 &= \int A(x,\grad u)\cdot \grad \eta \, dx \\
&= \int \varphi^p A(x,\grad u)\cdot \grad u + p\varphi^{p-1} (u-u_{Q_{2 r}}) A(x,\grad u)\cdot \grad \varphi\, dx. 
\end{align*} 
This implies that 
\begin{align*}
\int \varphi^p  \vert \grad u\vert^p\,dx \le 3 K^{12} \int  \varphi^{p-1}(u-u_{Q_{2 r}}) \vert \grad u\vert^{p-1} \vert \grad \varphi\vert\, dx. 
\end{align*} 
Next, we apply the Young inequality with $\eps$ and dual exponents $\frac{p-1}{p} + \frac{1}{p} = 1$ to get 
\begin{align*} 
\int \varphi^p  \vert \grad u\vert^p\,dx &\le \frac 1 2  \int \varphi^p \vert \grad u\vert^p \, dx + \frac{(p-1)^{p-1}}{p^p}(6K^{12})^{p-1} \int \vert u - u_{Q_{2 r}}\vert^p \vert \grad \varphi\vert^p\, dx\\
& \le \frac 1 2 \int \varphi^p \vert \grad u\vert^p\, dx + 36 K^{24} \int \vert u-u_{Q_{2 r}}\vert^p \vert \grad \varphi\vert^p\, dx. 
\end{align*} 
Thus we have 
\[
\int \varphi^p \vert \grad u\vert^p\, dx \le C(K) \int \vert u-u_{Q_{2 r}}\vert^p \vert \grad \varphi\vert^p\, dx. 
\]
Finally, this yields 
\[
\int_{Q_r} \vert \grad u\vert^p\, dx \le \frac{C(K)}{r^p} \int_{Q_{2 r}} \vert u-u_{Q_{2 r}}\vert^p\, dx,
\]
as needed. 
\end{proof} 

Next we apply the Poincar\'e-Sobolev inequality to deduce a reverse H\"older inequality. 

\begin{prop}
Assume that $Q_r \subset Q_{2 r} \subset\subset \an$ for some concentric cubes $Q_r$ and $Q_{2 r}$. Then there is an estimate 
\[
\left[\frac{1}{\vert Q_r\vert} \int_{Q_r} \vert \grad u\vert^p\, dx \right]^{1/p} \le {C(K)} \left[\frac{1}{\vert Q_{2 r}\vert} \int_{Q_{2 r}} \vert \grad u\vert^q\, dx\right]^{1/q},
\]
where $q = \frac{3p}{3+p}$. 
\end{prop}

\begin{proof}
Let $q = \frac{3p}{3+p}$ so that $1 < q < p \le 3$ and  
\[
q^* = \frac{3q}{3-q} = p. 
\]
The Poincar\'e-Sobolev inequality implies that 
\[
\left[\int_{Q_{2 r}} \vert u - u_{Q_{2 r}}\vert^{p} \, dx\right]^{1/p} \le C(q,r) \left[ \int_{Q_{2 r}} \vert \grad u\vert^q \, dx\right]^{1/q}.
\]
In fact, a simple scaling argument shows that $C(q,r) = C(q)$ does not depend on $r$. Moreover, since $q$ belongs to the compact range $[\frac 6 5, \frac 3 2]$, the constants $C(q)$ are uniformly bounded above by some constant $C$ that does not depend on $q$. 

Combining this with the Caccioppoli inequality, we deduce that 
\begin{align*}
\left[\int_{Q_r} \vert \grad u\vert^p \, dx\right]^{1/p} &\le \frac{C(K)}{r} \left[ \int_{Q_{2 r}} \vert u - u_{Q_{2 r}}\vert^p\, dx\right]^{1/p}  \\
&\le \frac{C(K)}{r} \left[\int_{Q_{2 r}} \vert \grad u\vert^q\, dx\right]^{1/q}. 
\end{align*}
Finally, multiplying both sides by $r^{-3/p}$ and using $\frac 3 q - \frac 3 p = 1$, we get 
\[
\left[\frac{1}{\vert Q_r\vert} \int_{Q_r} \vert \grad u\vert^p\, dx \right]^{1/p} \le {C(K)} \left[\frac{1}{\vert Q_{2 r}\vert} \int_{Q_{2 r}} \vert \grad u\vert^q\, dx\right]^{1/q},
\]
as desired. 
\end{proof}

The last step is to apply Gehring's lemma.    The following form of Gehring's lemma is taken from \cite[Proposition 6.1]{iwaniec1998gehring}. 

\begin{lemma}
Let $\Omega$ be a cube in Euclidean space.   Fix some $1 < b < \infty$. Assume that $f\in L^b(\Omega)$ is a non-negative function which satisfies 
\[
\left(\frac{1}{\vert Q\vert} \int_Q f^b\, dx\right)^{1/b} \le \frac{A}{\vert 2Q\vert} \int_{2Q} f\, dx
\]
for all cubes $Q \subset 2Q \subset \Omega$. Then for every 
\[
b < s < b + \frac{b-1}{10^{3+b} \cdot 4^3 \cdot A^b}
\]
we have 
\[
\left(\frac{1}{\vert \frac 1 2 \Omega\vert} \int_{\frac 1 2 \Omega} f^s \, dx\right)^{1/s} \le {100^3} 2^{\frac 3 s + \frac 3 b} \left(\frac{1}{\vert \Omega\vert} \int f^b\, dx\right)^{1/b}. 
\]
\end{lemma} 

We can now prove the main result of this section.  Let $U$ be an open set with smooth boundary which is compactly contained in $\an$. 

\begin{prop} \label{prop-improvement}
There are constants $C > 0$ and $\eps >0$ that depend only on $K$ and $U$ such that 
\[
\left(\int_{U} \vert \grad u\vert^{p+\eps}\, dx\right)^{1/(p+\eps)} \le C \left(\int_{\an} \vert \grad u\vert^p\, dx\right)^{1/p}.
\]
\end{prop} 

\begin{proof}
Fix a cube $\Omega \subset\subset \an$. Consider a cube $Q$ with $Q \subset 2Q \subset \Omega$.  From above, there is a reverse H\"older inequality 
\[
\left(\frac{1}{\vert Q\vert}\int_Q \vert \grad u\vert^p\, dx\right)^{1/p} \le C(K) \left(\frac{1}{\vert 2Q\vert} \int_{2Q} \vert \grad u\vert^q\, dx\right)^{1/q} 
\]
for all such cubes $Q$. Without loss of generality, we can suppose that $C(K) > 1$.  Let $f = \vert \grad u\vert^q$ and let $b = \frac p q = \frac{3+p}{3}\in [\frac 5 3,2]$. Then we have 
\[
\left(\frac{1}{\vert Q\vert} \int_Q f^b \,dx\right)^{1/b} \le \frac{C(K)}{\vert 2Q\vert} \int_{2Q} f\, dx
\]
for all such cubes $Q$.  We note the crude bound 
\[
\frac{b-1}{10^{3+b} \cdot 4^3 \cdot C(K)^b} \ge \frac{b}{p \cdot 10^8 \cdot C(K)^2}. 
\]
Thus we can apply Gehring's lemma to deduce that 
\[
\left(\frac{1}{\vert \frac 1 2 \Omega\vert} \int_{\frac 1 2 \Omega} \vert \grad u\vert^{qs} \, dx\right)^{1/(qs)} \le A \left(\frac{1}{\vert \Omega\vert} \int_\Omega \vert \grad u\vert^p\, dx\right)^{1/p}
\]
for $b < s < b\left(1 + \frac{1}{p \cdot 10^8 \cdot  C(K)^2}\right)$ and $A = 100^3 \cdot 2^6$. In particular, we have 
\[
\left(\frac{1}{\vert \frac 1 2 \Omega\vert} \int_{\frac 1 2 \Omega} \vert \grad u\vert^{p + \eps} \, dx\right)^{1/(p+\eps)} \le A \left(\frac{1}{\vert \Omega\vert} \int_\Omega \vert \grad u\vert^p\, dx\right)^{1/p}
\]
for $\eps = \frac{1}{10^8 \cdot C(K)^2}$. Finally, this local inequality implies that 
\[
\left(\int_{U} \vert \grad u\vert^{p+\eps}\, dx\right)^{1/(p+\eps)} \le C \left(\int_{\an} \vert \grad u\vert^p\, dx\right)^{1/p} 
\]
where $C$ and $\eps$ depend only on $K$ and $U$. 
\end{proof}

\section{Convergence in Measure}

In this section, we prove the main theorem about convergence in measure. 

\begin{theorem}
    Assume $M^3$ is a closed three manifold. Assume that smooth metrics $g_k$ on $M$ converge in measure to a smooth metric $g$. Further assume that there is a constant $K > 1$ such that the identity maps $(M,g_k)\to (M,g)$ are all $K$-bi-Lipschitz. Then if each metric $g_k$ has non-negative scalar curvature so does $g$. 
\end{theorem}

\begin{proof}
    Assume for contradiction that each metric $g_k$ has non-negative scalar curvature, but there is a point $q\in M$ such that $R_g(q) < 0$. Introduce geodesic normal coordinates $x$ near $q$ and introduce the rescaled metrics $g_r(x) = g(rx)$ and $g_{k,r}(x) = g_k(rx)$ defined on $\an = B(0,16)-B(0,1/2)$. Note that the metrics $g_{k,r}$ will all be $2K$-bi-Lipschitz to the Euclidean metric on $\an$ for small enough $r$.  
    
    Choose $p<3$ close enough to 3 so that $p+\eps > 3$, where $\eps$ is the improvement of integrability constant from Proposition \ref{prop-improvement} (applied with the $2K$ bi-Lipschitz bound and with the sets $U = \an(1)$ and $U=\an(2)$). Let $u_r$ be the $p$-harmonic function on $\an$ in the $g_r$ metric which is equal to $\vert x\vert^{-a}$ on the boundary of $\an$. Let $D_{p,r}$ be the annular quantity associated to $u_r$ in the $g_r$ metric. Since $R_{g}(q) < 0$, by the results of Section \ref{Section-Expansion}, we can fix some $r$ small enough that 
    \[
    2^a  D_{p,r}(2) <  D_{p,r}(1). 
    \]
    On the other hand, let $u_{k,r}$ be the $p$-harmonic function on $\an$ in the $g_{k,r}$ metric which equals $\vert x\vert^{-a}$ on the boundary of $\an$. Let $D_{p,k,r}$ be the annular quantity associated to $u_{k,r}$ in the $g_{k,r}$ metric. Since $g_{k,r}$ has non-negative scalar curvature, the monotonicity results in Section \ref{Section-Monotonicity} imply that 
    \[
    2^a  D_{p,k,r}(2) \ge  D_{p,k,r}(1).
    \]
    Thus we will obtain a contradiction provided we can show $D_{p,k,r}(i)\to D_{p,r}(i)$ as $k\to \infty$ for $i=1,2$. For simplicity, we will prove this for $i = 1$. The proof for $i=2$ is exactly the same.

     To establish this, we first claim that $u_{k,r}\to u_r$  in $W^{1,p}(\an)$.  For convenience, we drop the subscript $r$ in the remainder of the proof. Define 
    \begin{gather*}
        A(x,\xi) = \sqrt{\det g} (g^{\ell m}\xi_\ell\xi_m)^{\frac{p-2}{2}}g^{ij}\xi_i\bd_j,\\
        A_k(x,\xi) = \sqrt{\det g_k} (g_k^{\ell m}\xi_\ell\xi_m)^{\frac{p-2}{2}}g_k^{ij}\xi_i\bd_j.
    \end{gather*}
    By the bi-Lipschitz bounds, one obtains  
    \begin{gather*}
    \vert A(x,\xi)\vert\le C \vert \xi\vert^{p-1},\\
        \vert A_k(x,\xi)\vert \le C \vert \xi\vert^{p-1},\\
         \big(A_k(x,\xi) - A_k(x,\eta)\big)\cdot (\xi - \eta) \ge c \vert \xi - \eta\vert^p 
    \end{gather*}
    for some uniform constants $C > 0$ and $c > 0$. We now test the $p$-harmonic equation for $u$ and $u_k$ against the function $u_k - u$: 
    \[
    \int_{\an} A(x,\grad u)\cdot \grad(u_k-u)\, dx = 0, \quad \int_{\an} A_k(x,\grad u_k)\cdot \grad (u_k-u)\, dx = 0.
    \]
    Here the gradient and dot product are the Euclidean gradient and dot product. This implies that 
    \begin{align*}
    \int_{\an} \left[A(x,\grad u) - A_k(x,\grad u)\right]\cdot \grad (u_k-u)\, dx &= \int_{\an} \left[A_k(x,\grad u_k) - A_k(x,\grad u)\right]\cdot \grad(u_k-u)\, dx\\
    &\ge c \int_{\an} \vert \grad u_k - \grad u\vert^p \, dx. 
    \end{align*} 
    On the other hand, we have 
    \begin{align*} 
    &\int_{\an} \left[A(x,\grad u) - A_k(x,\grad u)\right]\cdot \grad (u_k-u)\, dx  \\
    &\qquad \le  \left[ \int_{\an} \vert A(x,\grad u) - A_k(x,\grad u)\vert^{\frac{p}{p-1}} \, dx\right]^{(p-1)/p}\left[\int_{\an} \vert \grad u_k - \grad u\vert^p\, dx \right]^{1/p}.
    \end{align*}
    Combined with the previous inequality, this yields 
    \[
    \int_{\an} \vert \grad u_k - \grad u\vert^p\, dx \le \int_{\an} \vert A(x,\grad u) - A_k(x,\grad u)\vert^{\frac{p}{p-1}}\, dx. 
    \]
    Finally, the right hand side goes to 0 as $k\to \infty$. Indeed, since $g_k \to g$ in measure, we can pass to a subsequence for which $g_k \to g$ pointwise almost-everywhere in $\an$. This implies that $A_k(x,\grad u) \to A(x,\grad u)$ pointwise almost-everywhere. Moreover,  
    \[
    \vert A(x,\grad u) - A_k(x,\grad u)\vert^{\frac{p}{p-1}} \le C \vert \grad u\vert^p 
    \]
    and the right hand side is integrable. Hence the integral on the right goes to 0 by dominated convergence. Thus we obtain $\grad u_k \to \grad u$ in $L^p$. Since $u_k = u$ on the boundary of $\an$, the Poincare inequality implies that $u_k \to u$ in $L^p$ as well. Thus we have $u_k\to u$ in $W^{1,p}$. 

    Next we upgrade this to convergence in $W^{1,s}(\an(1))$ for some $s > 3$.  By the improvement of integrability, we have 
    \begin{align*}
        \left[\int_{\an(1)} \vert \grad u_k\vert^{p+\eps}\, dx \right]^{1/(p+\eps)} \le C \left[\int_{\an} \vert \grad u_k\vert^p\, dx\right]^{1/p} 
    \end{align*}
    for $i = 1,2$. 
    Since $u_k\to u$ in $W^{1,p}(\an)$, the right hand side is uniformly bounded independent of $k$. It follows that $\grad u_k - \grad u$ is uniformly bounded in $L^{p+\eps}(\an(1))$. 
    Since $\grad u_k\to \grad u$ in $L^p(\an(1))$ and $\grad u_k - \grad u$ is uniformly bounded in $L^{p+\eps}(\an(1))$, it follows from the interpolation inequalities that $\grad u_k \to \grad u$ strongly in $L^s(\an(1))$ for every $p < s < p+\eps$. 
    In particular, we have now shown that $u_k \to u$ in $W^{1,s}(\an(1))$ for some $s > 3$.

    Finally, we claim that $\mathcal D_{p,k}(1)\to D_{p}(1)$.  Observe that $u_k\to u$ in $C^0$ and so $u_k \ge c > 0$ for all large enough $k$. Thus $\varphi(1,u_k) \to \varphi(1,u)$ pointwise and $\vert \varphi(1,u_k)\vert \le C$ for a uniform constant $C$. Again, by passing to a subsequence, we can suppose that $g_k \to g$ almost-everywhere and that $\grad u_k\to \grad u$ almost-everywhere. We have 
    \begin{gather*}
     D_{p,k}(1) = c_\psi + \int_{\an(1)} \varphi(1,u_k) \vert \grad^{g_k} u_k\vert_{g_k}^3\, dv_{g_k},\\
     D_p(1) = c_\psi + \int_{\an(1)} \varphi(1,u)\vert \grad^g u\vert_g^3 \, dv_g. 
    \end{gather*} 
    It follows that 
    \begin{align*}
        D_{p,k}(1) - D_p(1) &= \int_{\an(1)} \varphi(1,u_k)\sqrt{\det g_k} \vert \grad^{g_k} u_k\vert_{g_k}^3 - \varphi(1,u)\sqrt{\det g} \vert \grad^g u\vert_g^3\, dx\\
        &= \int_{\an(1)} \varphi(1,u_k)\sqrt{\det g_k} \vert \grad^{g_k} u_k\vert_{g_k}^3 - \varphi(1,u_k)\sqrt{\det g_k} \vert \grad^g u\vert_g^3\, dx\\ 
        &\qquad + \int_{\an(1)} \varphi(1,u_k)\sqrt{\det g_k} \vert \grad^{g} u\vert_g^3 - \varphi(1,u)\sqrt{\det g} \vert \grad^g u\vert_g^3\, dx.
    \end{align*}
    Observe that 
    \[
    \int_{\an(1)} \varphi(1,u_k)\sqrt{\det g_k} \vert \grad^{g} u\vert_g^3 - \varphi(1,u)\sqrt{\det g} \vert \grad^g u\vert_g^3\, dx \to 0
    \]
    by dominated convergence. Also observe that 
    \begin{align*}
        &\left\vert \int_{\an(1)} \varphi(1,u_k)\sqrt{\det g_k} \vert \grad^{g_k} u_k\vert_{g_k}^3 - \varphi(1,u_k)\sqrt{\det g_k} \vert \grad^g u\vert_g^3\, dx\right\vert \\
        & \qquad \le C \int_{\an(1)} \left\vert \vert \grad^{g_k} u_k\vert_{g_k}^3 - \vert \grad^g u\vert_g^3\right\vert\, dx \\
        &\qquad \le C \int_{\an(1)} \left\vert \vert \grad^{g_k} u_k\vert_{g_k}^3 - \vert \grad^{g_k} u\vert^3_{g_k}\right\vert + \left\vert \vert \grad^{g_k} u\vert^3_{g_k} - \vert \grad^g u\vert^3_g\right\vert\, dx.
    \end{align*}
    We again have 
    \[
    \int_{\an(1)} \left\vert \vert \grad^{g_k} u\vert_{g_k}^3 - \vert \grad^g u\vert_g^3\right\vert\, dx \to 0
    \]
    by dominated convergence. We further have 
    \begin{align*}
        &\int_{\an(1)} \left\vert \vert \grad^{g_k} u_k\vert^3_{g_k} - \vert \grad^{g_k} u\vert^3_{g_k} \right\vert\, dx\\
        &\qquad \le \int_{\an(1)} (\vert \grad^{g_k} u_k\vert^2_{g_k} + \vert \grad^{g_k} u_k\vert_{g_k} \vert \grad^{g_k} u\vert_{g_k} + \vert \grad^{g_k} u\vert^2_{g_k}) \vert \grad^{g_k} u_k - \grad^{g_k} u\vert_{g_k} \, dx\\
        &\qquad \le C \left[\int_{\an(1)} (\vert \grad u_k\vert^2 + \vert \grad u_k\vert \vert \grad u\vert + \vert \grad u\vert^2)^{3/2}\right]^{2/3} \left[\int_{\an(1)} \vert \grad u_k-\grad u\vert^3\, dx\right]^{1/3}\\
        &\qquad \le C \left[\int_{\an(1)} \vert \grad u_k-\grad u\vert^3\, dx\right]^{1/3}
    \end{align*}
    and so this also goes to 0. Combining everything, we deduce that $D_{p,k}(1)\to D_p(1)$ and this completes the proof. 
\end{proof}

\section{Non-Collapsed Convergence in \texorpdfstring{$L^p$}{Lp}}

In this final section, we note that the arguments of the paper also imply that scalar curvature lower bounds are preserved by non-collapsed convergence of the metric together with its inverse in $L^p$.  Here, following \cite{lee2023dp}, we use Perelman's entropy and Perelman's $\nu$-functional \cite{perelman2002entropy} to measure non-collapsing. 
The goal is to prove Theorem \ref{main-lp}.

The proof is essentially the same as before, but using the $d_P$-theory \cite{lee2023dp} to obtain the improvement of integrability.  Fix some $Q > 3$.  Let $g$ be a smooth metric on $M^3$ and assume that $R(g) \ge -\delta$ and $\nu(g,2)\ge -\delta$. We collect the following properties of the $d_Q$-metric which hold once $\delta$ is small enough. 

\begin{prop}
\label{cutoff} 
Fix a number $Q > 3$. Consider any metric $g$ satisfying $R(g) \ge -\delta$ and $\nu(g,2)\ge -\delta$.  Assuming $\delta$ is small enough (depending on $Q$), the following property is true. For every $x\in M$ and $0 < r \le 1$, there exists a function $\varphi$ such that $\varphi\equiv 1$ on $\mathcal B_{Q,g}(x,r)$, and $\varphi \equiv 0$ outside $\mathcal B_{Q,g}(x,2r)$, and 
\[
\left[\int_{\mathcal B_{Q,g}(x,2r)} \vert \grad \varphi\vert^Q\, dv_g\right]^{1/Q} \le \frac{C}{r}. 
\]
\end{prop}

\begin{proof} The existence of such a cutoff function is demonstrated in the proof of Lemma 8.13 in \cite{lee2023dp}. In particular, see \cite[Equation (8.30)]{lee2023dp} and note that the volume of $\mathcal B_{Q,g}(x,2r)$ is proportional to $r^{3Q/(Q-3)}$. 
\end{proof}

\begin{prop}
Fix $Q > 3$ and fix $1 < q < s < q^*$. Then assuming $\delta$ is small enough (depending on $Q$, $q$ and $s$), any metric $g$ with $R(g) \ge -\delta$ and $\nu(g,2) \ge -\delta$ satisfies a uniform Sobolev inequality 
\[
\inf_{c\in \R} \left[\int_{\mathcal B_{Q,g}(x,r)} \vert f - c\vert^s\, dv_g \right]^{1/s} \le C r^{\frac{Q}{Q-3}\left[1+\frac 3 s -\frac 3 q\right]} \left[\int_{\mathcal B_{Q,g}(x,2r)} \vert \grad f \vert^q\, dv_g\right]^{1/q}
\]
for all $x$ and all $0 < r \le 1$ and all functions $f \in W^{1,q}$. Here the constant $C = C(Q,q,s) > 0$ depends only on $Q$, $q$, and $s$. 
\end{prop}

\begin{proof}
First suppose that $r = 1$.  Assuming $\delta$ is small enough, we can apply the $\eps$-regularity theorem \cite[Theorem 1.7]{lee2023dp} to obtain a set $\Omega$ containing $x$ and a diffeomorphism $\psi\colon \Omega\to B(0,100)$ with $\psi(x) = 0$. Here $B(0,100)$ denotes a Euclidean ball of radius 100.  In the following, we identify $\Omega$ with $B(0,100)$ and suppress the diffeomorphism $\psi$.  Again assuming $\delta$ is sufficiently small, one has 
$
\mathcal B_{Q,g}(x,1) \subset \mathcal B_{Q,euc}(0,3/2) \subset B_{Q,g}(x,2).
$

Now consider a function $f \in W^{1,q}(\mathcal B_{Q,g}(x,2))$. 
Since $q < s < q^*$, we can choose $\kappa > 1$ close enough to 1 that 
\[
\frac{q}{\kappa} < s\kappa < \left(\frac{q}{\kappa}\right)^*. 
\]
Note that $f$ restricts to a $W^{1,q/\kappa}$ function on $\mathcal B_{Q,euc}(0,3/2)$. 
 Hence the ordinary Euclidean Sobolev-Poincare inequality gives 
\[
\inf_{c\in \R} \left[\int_{\mathcal B_{Q,\text{euc}}(0,3/2)} \vert f - c\vert^{\kappa s}\, dv_{\text{euc}} \right]^{1/(\kappa s)} \le C  \left[\int_{\mathcal B_{Q,\text{euc}}(0,3/2)} \vert \grad^{euc} f \vert^{q/\kappa} \, dv_{\text{euc}}\right]^{\kappa/q}.
\]
Assuming $\delta$ is small enough, the $L^Q$-coefficient control \cite[Theorem 1.11]{lee2023dp} then implies that 
\[
\inf_{c\in \R} \left[\int_{\mathcal B_{Q,\text{euc}}(0,3/2)} \vert f - c\vert^{s}\, dv_{g} \right]^{1/s} \le C  \left[\int_{\mathcal B_{Q,\text{euc}}(0,3/2)} \vert \grad^{g} f \vert^{q} \, dv_{g}\right]^{1/q}
\]
for some slightly larger constant $C$. Finally, this yields 
\[
\inf_{c\in \R} \left[\int_{\mathcal B_{Q,g}(x,1)} \vert f - c\vert^{s}\, dv_{g} \right]^{1/s} \le C  \left[\int_{\mathcal B_{Q,g}(x,2)} \vert \grad^{g} f \vert^{q} \, dv_{g}\right]^{1/q}
\]
and the proof is complete for $r = 1$. 

The general case now follows from a scaling argument. Indeed, the hypotheses on scalar curvature and entropy improve upon zooming in.  Moreover, defining $\tilde g = \rho^{-2} g$ one has 
\[
\mathcal B_{Q,\tilde g}(x,1) = \mathcal B_{Q,g}(x,\rho^{1-3/Q}).
\]
Hence the result follows by applying the special case above to $\tilde g$ with $\rho^{1-3/Q} = r$ and then re-writing everything back in terms of the metric $g$. 
\end{proof}

\begin{prop}
\label{improvement}
Fix a number $Q > 3$.  If $\delta > 0$ is sufficiently small (depending on $Q$) then the following property holds.  Let $g$ be any metric with $R(g) \ge -\delta$ and $\nu(g,2)\ge -\delta$.  Fix any $2 < p < 3$. Assume that $u$ is a $p$-harmonic function on $\mathcal B_{Q,g}(x,8)$. There are constants $\eps = \eps(Q) > 0$ and $C = C(Q) > 0$ such that 
\[
\left[\int_{\mathcal B_{Q,g}(x,1)} \vert \grad u\vert^{p+\eps}\, dv_g\right]^{{1}/{(p+\eps)}} \le C \left[\int_{\mathcal B_{Q,g}(x,4)} \vert \grad u\vert^p\, dv_g\right]^{1/p}.
\]
\end{prop}

\begin{proof}
Fix any ball $\mathcal B_{Q,g}(y,4r) \subset \mathcal B_{Q,g}(x,8)$. Let $\varphi$ be the good cutoff function on $\mathcal B_{Q,g}(y,2r)$ constructed in Proposition \ref{cutoff}. We test the $p$-harmonic equation with the function $\eta = \varphi^p (u-c)$ where $c$ is a fixed constant. Then applying standard arguments we get the Caccioppoli inequality 
\[
\int_{\mathcal B_{Q,g}(y,r)} \vert \grad u\vert^p\, dv_g \le \int_{\mathcal B_{Q,g}(y,2r)} \vert u - c\vert^p \vert \grad \varphi\vert^p\, dv_g. 
\]
Define 
\[
s = \frac{Qp}{Q-p}
\]
so that 
\[
\frac{ps}{s-p} = Q. 
\]
We note that $p < s$. In fact, it is possible to choose $q$ so that $q < p < s < q^*$. Indeed, set 
\[
t = \frac{3s}{3+s} 
\]
so that $t^* = s$. Then it suffices to note that $Q > 3$ implies $t < p$ and so any choice $t < q < p$ will do.  Now we apply H\"older's inequality using 
\[
\frac{p}{s} +\frac{s-p}{s} = 1
\]
to get 
\[
\int_{\mathcal B_{Q,g}(y,2r)} \vert u-c\vert^p \vert \grad \varphi\vert^p\, dv_g \le \left[\int_{\mathcal B_{Q,g}(y,2r)} \vert u-c\vert^s\, dv_g\right]^{p/s} \left[\int_{\mathcal B_{Q,g}(y,2r)} \vert \grad \varphi\vert^{Q} \, dv_g\right]^{p/Q}.
\]
Combined with the previous estimate and the properties of the good test function, we get 
\[
\left[\int_{\mathcal B_{!,g}(y,r)} \vert \grad u\vert^p\, dv_g\right]^{1/p}\le \frac{C}{r}\left[\int_{\mathcal B_{Q,g}(y,2r)} \vert u-c\vert^s\, dv_g\right]^{1/s}.
\]
Now take the infimum over all constants $c$ and apply the uniform Sobolev inequality to get 
\[
\left[\int_{\mathcal B_{Q,g}(y,r)} \vert \grad u\vert^p\, dv_g\right]^{1/p} \le C r^{\frac{Q}{Q-3}\left[1 + \frac 3 s - \frac 3 q\right]-1} \left[\int_{\mathcal B_{Q,g}(y,4r)} \vert \grad u\vert^q \, dv_g\right]^{1/q}. 
\]
Finally, note that the volume $\vert\mathcal B_{Q,g}(y,r)\vert$ is uniformly comparable to $r^{3Q/(Q-3)}$. Hence the above inequality can be simplified to give the reverse H\"older inequality 
\begin{align*}
\left[\frac{1}{\vert \mathcal B_{Q,g}(y,r)\vert} \int_{\mathcal B_{Q,g}(y,r)} \vert \grad u\vert^p\, dv_g\right]^{1/p} \le C \left[\frac{1}{\vert \mathcal B_{Q,g}(y,4r)\vert}\int_{\mathcal B_{Q,g}(y,4r)} \vert \grad u\vert^q\, dv_g\right]^{1/q}. 
\end{align*} 
The result now follows from Gehring's lemma since $d_Q$ is doubling up to scale one; see \cite{maasalo2007gehring} and the proof of Lemma 8.13 in \cite{lee2023dp}. Since the numbers $p$, $q$, and $s$ can be chosen in some uniform compact ranges, the doubling constant is uniform, and the reverse H\"older constant is uniform, it follows that the constants $C$ and $\eps$ will depend only on $Q$. 
\end{proof}

We can now proceed with the proof of Theorem \ref{main-lp}.  
 Assume for contradiction that there is a point $y\in M$ such that $R_g(y) < 0$. Let $\an = B(0,16) - B(0,1/2)$ and let $x$ be a geodesic normal coordinate system around $y$ in the $g$ metric.  Let $\eps$ be the improvement of integrability constant from Proposition \ref{improvement} applied with $Q = 10$ and choose $p < 3$ so that $p+\eps > 3$.   Considering $g_r(x) = g(rx)$ and $g_{k,r}(x) = g_k(rx)$ for a sufficiently small $r > 0$, we can suppose as before that $D_p(1) > 2^a D_p(2)$ where $D_p$ is the function associated to the $p$-harmonic function $u_r$ in the $g_r$ metric which equals $\vert x\vert^{-a}$ on the boundary of $\an$.  We can further shrink $r$ so that $g_r$ is smoothly close to $g_{euc}$ on $B(0,100)$. 
In the rest of the proof, we drop the subscript $r$ and write $g$ and $g_k$ instead of $g_r$ and $g_{k,r}$. 
Let $u_k$ be the $p$-harmonic function in the $g_k$ metric which is equal to $\vert x\vert^{-a}$ on the boundary of $\an$. Then since $g_k$ has non-negative scalar curvature, we have $D_{p,k}(1) \le 2^a D_{p,k}(2)$ for all $k$. Therefore, as before, to get a contradiction it suffices to show that $D_{p,k}(1)\to D_p(1)$ and $D_{p,k}(2)\to D_p(2)$. 

Thus we first aim to show that $u_k\to u$ in a suitable sense. We identify $g$ and $g_k$ with metrics on $B(0,100)$. We note that the $L^P$ convergence implies that  
\begin{gather*}
\int_{B(0,100)} \vert g-g_k\vert^P_{h}\, dv_h \to 0, \quad \int_{B(0,100)} \vert g^{-1}-g_k^{-1}\vert^P_{h}\, dv_h \to 0,\\
\int_{B(0,100)} \vert g g_k^{-1}-I\vert^P_{h} \, dv_h\to 0, \quad \int_{B(0,100)} \vert g_k g^{-1} - I\vert^P_{h}\, dv_h \to 0
\end{gather*}
for any fixed smooth reference metric $h$. 

\begin{prop}
\label{comparison}
Given any $\kappa> 1$, one can choose $P$ sufficiently large so that 
\begin{gather*}
C^{-1} \|f\|_{L^{q/\kappa}(g_k)} \le \|f\|_{L^q(g)}\le C \|f \|_{L^{\kappa q}(g_k)},\\
C^{-1} \|\grad^{g_k} f\|_{L^{q/\kappa}(g_k)} \le \|\grad^g f\|_{L^q(g)}\le C \|\grad^{g_k} f \|_{L^{\kappa q}(g_k)}
\end{gather*}
for all Sobolev $f$ and all $q\in [2,6]$.  Here $C > 1$ is a constant that does not depend on $f$, $q$, or  $k$. 
\end{prop}

\begin{proof} 
First we study the volume forms.  Take $h$ to be the background Euclidean metric and write the metrics $g_k$ and $g$ in coordinates.  In terms of the Frobenius norm
\[
\vert A\vert = (\lambda_1^2 + \lambda_2^2 + \lambda_3^2)^{1/2}
\]
on positive definite symmetric matrices, one has the elementary inequality 
\[
 \det A \le C \vert A\vert^{3/2}. 
\]
Hence one has 
\[
\int (\det g_k)^{2P/3} \, dx \le C \int \vert g_k\vert^{P} \, dx \le C
\]
since $g_k$ is uniformly bounded in $L^P$.
Similarly, $\det g$, $\det g_k^{-1}$, and $\det g^{-1}$ all belong to $L^{2P/3}$ with uniform bounds. 

Now consider a Sobolev function $f$.  Observe that 
\begin{align*}
\left[\int \vert f\vert^q\, dv_g\right]^{1/q} &= \left[\int \vert f\vert^q\, \frac{\sqrt{\det g}}{\sqrt{\det g_k}}  dv_{g_k}\right]^{1/q} \\
&\le \left[\int \vert f\vert^{\kappa q} \, dv_{g_k}\right]^{1/(\kappa q)} \left[ \int \left(\frac{\sqrt{ \det g}}{\sqrt{\det g_k}}\right)^{\frac{\kappa}{\kappa-1}}\, dv_{g_k}\right]^{\frac{\kappa-1}{\kappa q}}.
\end{align*}
Next, note that 
\begin{align*}
\int \left(\frac{\sqrt{ \det g}}{\sqrt{\det g_k}}\right)^{\frac{\kappa}{\kappa-1}}\, dv_{g_k} = \int (\sqrt{\det g})^{\frac{\kappa}{\kappa-1}} (\sqrt{\det g_k^{-1}})^{\frac{1}{\kappa-1}}\, dx.
\end{align*} 
By H\"older's inequality, the right hand side is uniformly bounded by a constant $C$ provided $P$ is large enough depending on $\kappa$. Thus we have obtained 
\[
\|f\|_{L^q(g)} \le C \|f\|_{L^{\kappa q}(g_k)}.
\]
The inequality 
\[
\|f\|_{L^{q/\kappa}(g_k)} \le C \|f\|_{L^q(g)}
\]
is obtained similarly. 

It remains to study the gradients.  With respect to the background Euclidean metric, one has 
\[
\vert \grad^{g_k} f\vert_{g_k}^2 = \la g_k^{-1} \grad f, \grad f\ra
\]
and so 
\[
\frac{1}{\vert g_k\vert} \vert \grad f\vert^2 \le \vert \grad^{g_k} f\vert_{g_k}^2 \le \vert g_k^{-1}\vert \vert \grad f\vert^2
\]
where we are now using the operator norm. 
Similar bounds hold for $g$ and it follows that 
\[
\vert \grad^g f\vert^2_g \le \vert g^{-1}\vert \vert \grad f\vert^2 \le \vert g_k\vert \vert g^{-1}\vert \vert \grad^{g_k} f\vert_{g_k}^2. 
\]
Hence we have 
\begin{align*}
\left[\int \vert \grad^g f\vert_g^q \, dv_g\right]^{1/q} \le \left[\int \vert \grad^{g_k} f\vert_{g_k}^q \vert g_k\vert^{q/2} \vert g^{-1}\vert^{q/2} \frac{\sqrt{\det g}}{\sqrt{\det g_k}}\, dv_{g_k}\right]^{1/q}. 
\end{align*}
As above, the inequality 
\[
\|\grad^g f\|_{L^q(g)} \le C \|\grad^{g_k} f\|_{L^{\kappa q}(g_k)} 
\]
now follows by applying H\"older's inequality provided $P$ is large enough depending on $\kappa$. The reverse inequality 
\[
\|\grad^{g_k} f\|_{L^{q\kappa}(g_k)}  \le C \|\grad^g f\|_{L^q(g)}
\]
is obtained similarly. 
\end{proof}

We now aim to demonstrate the convergence of $u_k$ to $u$.  Recall that $\eps > 0$ is the improvement of integrability constant supplied by Proposition \ref{improvement} with $Q = 10$.  Fix $\kappa > 1$ for which 
\[
p\kappa < 3 < \frac{p+\eps}{\kappa}
\]
 and then choose $P$ sufficiently large according to the previous proposition.  Recall that $\an(1) = B(0,4) - B(0,1)$. 

\begin{prop}
The quantity 
\[
\int_{\an} \vert \grad^{g_k} u_k\vert^p\, dv_k
\]
is uniformly bounded. 
\end{prop}

\begin{proof}
By the energy minimizing property of $p$-harmonic functions and Proposition \ref{comparison}, we have 
\[
\left[\int_{\an} \vert \grad^{g_k} u_k\vert^p\, dv_{g_k}\right]^{1/p} \le \left[\int_{\an} \vert \grad^{g_k} u\vert^p\, dv_{g_k}\right]^{1/p} \le C \left[\int_{\an} \vert \grad^g u\vert^{p\kappa} \, dv_g\right]^{1/(p\kappa)}. 
\]
Now $g$ is a smooth metric, so the $p$-harmonic function $u$ belongs to $C^{1,\alpha}(\cl \an)$. In particular, the quantity on the far right hand side above is finite. 
\end{proof} 

\begin{prop}
There is some $q > 3$ for which 
\[
\int_{\an(1)} \vert \grad u_k\vert^q\, dx
\]
is uniformly bounded. 
\end{prop}

\begin{proof}  By Proposition \ref{comparison}, we have 
\[
\left[\int_{\an(1)} \vert \grad^{g} u_k\vert^{(p+\eps)/\kappa}\, dv_{g}\right]^{\kappa/(p+\eps)}\le C\left[\int_{\an(1)} \vert \grad^{g_k} u_k\vert^{p+\eps}\, dv_{g_k}\right]^{1/(p+\eps)}.
\]
Recall that $Q = 10$.  Note that there is a uniform lower bound on the volume of balls $\vert \mathcal B_{Q,g_k}(y,1)\vert$. Hence the Vitali covering lemma implies that $\an(1)$ can be covered by a uniform number of unit $d_Q$-balls in the $g_k$ metric. 
Thus applying the improvement of integrability in the $g_k$ metric and using the previous proposition gives 
\[
\left[\int_{\an(1)} \vert \grad^{g_k} u_k\vert^{p+\eps}\, dv_{g_k}\right]^{1/(p+\eps)} \le C \left[\int_{\an} \vert \grad^{g_k} u_k\vert^p\, dv_{g_k}\right]^{1/p} \le C. 
\]
Combining everything, we deduce that 
\[
\left[\int_{\an(1)} \vert \grad^{g} u_k\vert^{(p+\eps)/\kappa}\, dv_{g}\right]^{\kappa/(p+\eps)} \le C 
\]
for a uniform constant $C$. Since $g$ is smoothly close to Euclidean, the result now follows with $q = (p+\eps)/\kappa > 3$. 
\end{proof}

\begin{prop}
The functions $u_k$ converge strongly to $u$ in $W^{1,s}(\an)$ for some $s < p$. 
\end{prop}

\begin{proof}
We note that for all vectors $\xi$ and $\eta$ one has 
\[
g_k(\vert \xi\vert^{p-2}_{g_k} \xi - \vert \eta\vert^{p-2}_{g_k} \eta, \xi-\eta) \ge 2^{2-p} \vert \xi - \eta\vert_{g_k}^p
\]
and likewise for $g$. By the $p$-harmonic equation, we have 
\[
\int  g_k(\vert \grad^{g_k} u_k\vert^{p-2}_{g_k} \grad^{g_k} u_k , \grad^{g_k}\eta)\, dv_k = 0, \quad \int  g(\vert\grad^{g} u\vert^{p-2}_{g} \grad^{g} u , \grad^{g}\eta)\, dv = 0 
\]
for all smooth, compactly supported $\eta$. This is equivalent to 
\[
\int \mathcal A_k(x,\grad u_k) \cdot \grad \eta \, dx = 0, \quad \int \mathcal A(x,\grad u) \cdot \grad \eta \, dx = 0
\]
where 
\[
A_k(x,\xi) = \sqrt{\det g_k} (g_k^{\ell m}\xi_\ell\xi_m)^{\frac{p-2}{2}} g_k^{ij} \xi_i \bd_j, \quad \mathcal A(x,\xi) = \sqrt{\det g} (g^{\ell m}\xi_\ell\xi_m)^{\frac{p-2}{2}} g^{ij} \xi_i \bd_j. 
\]
We test this with $\eta =  u_k -  u$ to get 
\begin{align*}
&\int (\mathcal A(x,\grad u) - \mathcal A_k(x,\grad u))\cdot (\grad u_k - \grad u)\, dx \\
&\qquad = \int (\mathcal A_k(x,\grad u_k)-\mathcal A_k(x,\grad u))\cdot (\grad u_k-\grad u)\, dx\\
&\qquad = \int g_k(\vert \grad^{g_k} u_k\vert^{p-2}_{g_k} \grad^{g_k} u_k - \vert \grad^{g_k} u\vert^{p-2}_{g_k} \grad^{g_k} u, \grad^{g_k}u_k - \grad^{g_k} u)\, dv_k \\
&\qquad \ge C \int \vert \grad^{g_k} u_k - \grad^{g_k} u \vert_{g_k}^p\, dv_k.
\end{align*} 
Then by Proposition \ref{comparison} and the fact that $g$ is smoothly close to Euclidean, we have 
\[
\int \vert \grad^{g_k} u_k - \grad^{g_k} u \vert_{g_k}^p\, dv_k \ge C\left[ \int \vert \grad u_k - \grad u\vert^{p/\kappa}\, dx\right]^\kappa.
\]
On the other hand, H\"older's inequality gives 
\begin{align*} 
&\int (\mathcal A(x,\grad u) - \mathcal A_k(x,\grad u))\cdot (\grad u_k - \grad u)\, dx\\
&\qquad \le \left[\int \vert \mathcal A(x,\grad u) - \mathcal A_k(x,\grad u)\vert^q \, dx \right]^{1/q} \left[\int \vert \grad u_k - \grad u\vert^{p/\kappa} \, dx\right]^{\kappa/p}
\end{align*} 
where 
\[
q = \frac{p}{p-\kappa}. 
\]
Combined with the previous inequality, this yields 
\[
\left[\int \vert \grad u_k - \grad u\vert^{p/\kappa}\, dx \right]^{\frac{\kappa(p-1)}{p}} \le \left[\int \vert \mathcal A(x,\grad u) - \mathcal A_k(x,\grad u)\vert^q\, dx\right]^{1/q}. 
\]
Hence to complete the proof it suffices to show that 
\[
\left[\int \vert \mathcal A(x,\grad u) - \mathcal A_k(x,\grad u)\vert^q \, dx\right]^{1/q}  \to 0
\]
as $k\to \infty$.  Observe that 
\begin{align*} 
&\left[\int \vert \mathcal A(x,\grad u) - \mathcal A_k(x,\grad u)\vert^q \, dx\right]^{1/q} \\
&\qquad =\left[ \int \vert \sqrt{\det g} \vert \grad^g u\vert^{p-2} \grad^g u - \sqrt{\det g_k} \vert \grad^{g_k} u\vert^{p-2} \grad^{g_k}u\vert^q\, dx\right]^{1/q} \\
&\qquad \le \left[\int \vert \sqrt{\det g} \vert \grad^g u\vert^{p-2} \grad^g u - \sqrt{\det g_k} \vert \grad^{g} u\vert^{p-2} \grad^{g}u\vert^q\, dx\right]^{1/q} \\
&\qquad \qquad + \left[\int \vert \sqrt{\det g_k} \vert \grad^g u\vert^{p-2} \grad^g u - \sqrt{\det g_k} \vert \grad^{g_k} u\vert^{p-2} \grad^{g_k}u\vert^q\, dx\right]^{1/q}.
\end{align*} 
For the first integral, note that $u$ is $C^1$ and so 
\begin{align*}
&\left[\int \vert \sqrt{\det g} \vert \grad^g u\vert^{p-2} \grad^g u - \sqrt{\det g_k} \vert \grad^{g} u\vert^{p-2} \grad^{g}u\vert^q\, dx\right]^{1/q} \\
&\qquad  \le C \left[\int \vert \sqrt{\det g} - \sqrt{\det g_k}\vert^q\, dx\right]^{1/q}.
\end{align*} 
The second line goes to 0 as $k\to \infty$ by the $L^P$ convergence of $g_k$ to $g$. 
 For the second integral, we have 
\begin{align*}
&\left[\int \vert \sqrt{\det g_k} \vert \grad^g u\vert^{p-2} \grad^g u - \sqrt{\det g_k} \vert \grad^{g_k} u\vert^{p-2} \grad^{g_k}u\vert^q\, dx\right]^{1/q}\\
&\qquad \le \left[\int (\sqrt{\det g_k})^{2q} dx\right]^{1/(2q)} \left[\int \vert  \vert \grad^g u\vert^{p-2} \grad^g u -  \vert \grad^{g_k} u\vert^{p-2} \grad^{g_k}u\vert^{2q}\, dx\right]^{1/(2q)}.
\end{align*}
The first term is bounded assuming $P$ is large enough. For the second, we can further estimate 
\begin{align*}
& \left[\int \vert  \vert \grad^g u\vert^{p-2} \grad^g u -  \vert \grad^{g_k} u\vert^{p-2} \grad^{g_k}u\vert^{2q}\, dx\right]^{1/(2q)}\\
&\qquad \le  \left[\int \vert  \vert \grad^g u\vert^{p-2} \grad^g u -  \vert \grad^{g_k} u\vert^{p-2} \grad^{g}u\vert^{2q}\, dx\right]^{1/(2q)}\\
&\qquad\qquad +\left[\int \vert  \vert \grad^{g_k} u\vert^{p-2} \grad^g u -  \vert \grad^{g_k} u\vert^{p-2} \grad^{g_k}u\vert^{2q}\, dx\right]^{1/(2q)}.
\end{align*} 
The first integral satisfies 
\[
\int \vert  \vert \grad^g u\vert^{p-2} \grad^g u -  \vert \grad^{g_k} u\vert^{p-2} \grad^{g}u\vert^{2q}\, dx\le C \int \vert  \vert \grad^g u\vert^{p-2}  -  \vert \grad^{g_k} u\vert^{p-2} \vert^{2q}\, dx.
\]
Now observe that 
\[
\vert \vert\grad^g u\vert^2_g - \vert \grad^{g_k} u\vert^2_{g_k}\vert = \vert \la (g^{-1}-g_k^{-1}) \grad u,\grad u\ra\vert \le \vert g^{-1} - g_k^{-1}\vert \vert \grad u\vert^2, 
\]
where we have used the operator norm. Since $\frac{p-2}{2} < 1$, we can use the elementary inequality 
\[
\vert a^{\frac{p-2}{2}} - b^{\frac{p-2}{2}}\vert \le \vert a-b\vert^{\frac{p-2}{2}} 
\]
to get 
\[
\vert (\vert \grad^g u\vert_g^2)^{\frac{p-2}{2}} - (\vert \grad^{g_k} u\vert_{g_k}^2)^{\frac{p-2}{2}}\vert \le \vert \vert \grad^g u\vert_g^2-\vert \grad^{g_k} u\vert_{g_k}^2\vert^{\frac{p-2}{2}} 
\]
Combining everything, it then follows that 
\begin{align*}
\int \vert  \vert \grad^g u\vert^{p-2}  -  \vert \grad^{g_k} u\vert^{p-2} \vert^{2q}\, dx &\le \int (\vert g^{-1} - g_k^{-1}\vert \vert \grad u\vert^2)^{q(p-2)}\, dx\\
&\le C \int \vert g^{-1} - g_k^{-1}\vert^{q(p-2)} \, dx
\end{align*}
and this goes to 0 provided $P$ is large enough.  We've now argued that 
\[
\left[\int \vert  \vert \grad^g u\vert^{p-2} \grad^g u -  \vert \grad^{g_k} u\vert^{p-2} \grad^{g}u\vert^{2q}\, dx\right]^{1/(2q)} \to 0. 
\]
Finally, consider the second integral. It satisfies 
\begin{align*}
&\int \vert  \vert \grad^{g_k} u\vert^{p-2} \grad^g u -  \vert \grad^{g_k} u\vert^{p-2} \grad^{g_k}u\vert^{2q}\, dx \\
&\qquad = \int   \vert \grad^{g_k} u\vert^{2q(p-2)} \vert \grad^g u -  \grad^{g_k}u\vert^{2q}\, dx\\
&\qquad \le \left[\int \vert \grad^{g_k} u\vert^{4q(p-2)} \, dx\right]^{1/2} \left[\int \vert \grad^g u - \grad^{g_k} u \vert^{4q}\right]^{1/2}. 
\end{align*}
Using the bound $\vert \grad^{g_k} u\vert^2 \le \vert g_k^{-1}\vert \vert \grad u\vert^2$, it follows that the first term is bounded if $P$ is large enough. Finally, we have 
\[
\grad^g u -\grad^{g_k} u = (g^{-1} - g_k^{-1})\grad u 
\]
and so 
\[
\int \vert \grad^g u - \grad^{g_k} u \vert^{4q} \le \int \vert g^{-1} - g_k^{-1}\vert^{4q} \vert \grad u\vert^{4q}\,dx \le C \int \vert g^{-1} - g_{k}^{-1}\vert^{4q}\, dx
\]
will go to zero provided $P$ is large enough. Putting everything together, we have now demonstrated that 
\[
\left[\int \vert \mathcal A(x,\grad u) - \mathcal A_k(x,\grad u)\vert^q \, dx\right]^{1/q}  \to 0
\]
and the proof is complete. 
\end{proof}

Combined with the uniform energy bound, it now follows by interpolation that $u_k\to u$ strongly in $W^{1,q}$ for some $q > 3$. This implies that $u_k\to u$ in $C^0$.  Finally, it remains to use this to show the convergence of $D_{k,p}$ to $D_k$.  For simplicity, we will show that $D_{k,p}(1) \to D_{p}(1)$, noting that $D_{k,p}(2) \to D_{p}(2)$ for exactly the same reason. This will complete the proof of Theorem \ref{main-lp}. 

\begin{prop}
We have $D_{k,p}(1)\to D_p(1)$ as $k\to \infty$. 
\end{prop}

\begin{proof}
Recall that 
\[
D_p(1) = c_\psi + \int_{\an(1)} \varphi(1,u) \vert \grad^g u\vert^3\, dv_g
\]
and likewise for $D_{p,k}(1)$. Thus we need to show that 
\[
\int_{\an(1)} \left\vert \varphi(1,u) \vert \grad^g u\vert^3 \sqrt{\det g} - \varphi(1,u_k) \vert \grad^{g_k} u_k\vert^3 \sqrt{\det g_k}\right\vert\, dx \to 0
\]
as $k\to \infty$. The integral on the left is at most 
\begin{align*}
&\int_{\an(1)} \left\vert \varphi(1,u) \vert \grad^g u\vert^3 \sqrt{\det g} - \varphi(1,u_k) \vert \grad^{g} u\vert^3 \sqrt{\det g}\right\vert\, dx \\
&\qquad + \int_{\an(1)} \left\vert \varphi(1,u_k) \vert \grad^g u\vert^3 \sqrt{\det g} - \varphi(1,u_k) \vert \grad^{g} u_k\vert^3 \sqrt{\det g}\right\vert\, dx\\
&\qquad + \int_{\an(1)} \left\vert \varphi(1,u_k) \vert \grad^g u_k\vert^3 \sqrt{\det g} - \varphi(1,u_k) \vert \grad^{g_k} u_k\vert^3 \sqrt{\det g}\right\vert\, dx\\
&\qquad + \int_{\an(1)} \left\vert \varphi(1,u_k) \vert \grad^{g_k} u_k\vert^3 \sqrt{\det g} - \varphi(1,u_k) \vert \grad^{g_k} u_k\vert^3 \sqrt{\det g_k}\right\vert\, dx.
\end{align*} 
The first integral goes to 0 since $\varphi(1,u_k)\to \varphi(1,u)$ in $C^0$ and $\vert \grad^g u\vert^3 \sqrt{\det g}$ is integrable. The second integral goes to 0 since 
\begin{align*} 
&\int_{\an(1)} \left\vert \varphi(1,u_k) \vert \grad^g u\vert^3 \sqrt{\det g} - \varphi(1,u_k) \vert \grad^{g} u_k\vert^3 \sqrt{\det g}\right\vert\, dx\\
&\qquad  \le C \int_{\an(1)} \vert  \vert \grad^g u\vert^3 - \vert \grad^g u_k\vert^3\vert\, dv_g
\end{align*} 
and $\grad u_k\to \grad u$ in $W^{1,3}$. 

 To see the third integral goes to zero, note that 
\begin{align*}
&\int_{\an(1)} \left\vert \varphi(1,u_k) \vert \grad^g u_k\vert^3 \sqrt{\det g} - \varphi(1,u_k) \vert \grad^{g_k} u_k\vert^3 \sqrt{\det g}\right\vert\, dx\\
&\qquad \le C \left[\int (\det g)^2\,d x\right]^{1/2} \left[\int \vert  \vert \grad^g u_k\vert^3 - \vert \grad^{g_k} u_k\vert^3\vert^2\, dx\right]^{1/2}.  
\end{align*}
The integral with $\det g$ is uniformly bounded. For the remaining integral, observe that 
\[
\vert (\grad^g u_k\vert^2)^{3/2} - (\vert \grad^{g_k}u_k\vert^2)^{3/2}\vert \le C \vert \vert\grad^g u_k\vert^2_g - \vert \grad^{g_k} u_k\vert_{g_k}^2\vert (\vert \grad^g u_k\vert_g + \vert \grad^{g_k}u_k\vert_{g_k})
\]
and that 
\begin{gather*}
\vert \vert\grad^g u_k\vert_g^2 - \vert \grad^{g_k} u_k\vert_{g_k}^2\vert \le \vert g^{-1}-g_k^{-1}\vert \vert \grad u_k\vert^2,\\
\vert \grad^g u_k\vert_{g} \le \vert g^{-1}\vert^{1/2} \vert \grad u_k\vert, \quad \vert \grad^{g_k} u_k\vert_{g_k} \le \vert g_k^{-1}\vert^{1/2} \vert \grad u_k\vert.
\end{gather*}
Thus we obtain 
\begin{align*}
\int \vert  \vert \grad^g u_k\vert^3 - \vert \grad^{g_k} u_k\vert^3\vert^2\, dx \le C \int \vert g^{-1} - g_k^{-1}\vert (\vert g^{-1}\vert^{1/2} + \vert g_k^{-1}\vert^{1/2}) \vert \grad u_k\vert^3\, dx.  
\end{align*} 
Since $u_k$ is bounded in $W^{1,q}$ for some $q > 3$, this will go to zero as $k\to \infty$ by H\"older's inequality provided $P$ is selected sufficiently large depending on $\eps$. Hence the third integral goes to 0. 

Finally consider the fourth integral. We have 
\begin{align*}
&\int_{\an(1)} \left\vert \varphi(1,u_k) \vert \grad^{g_k} u_k\vert^3 \sqrt{\det g} - \varphi(1,u_k) \vert \grad^{g_k} u_k\vert^3 \sqrt{\det g_k}\right\vert\, dx\\
&\qquad \le C \int \vert \grad^{g_k}u_k\vert^3 \vert \sqrt{\det g}-\sqrt{\det g_k}\vert\, dx \\
& \qquad \le C \int \vert g_k^{-1}\vert^{3/2} \vert \grad u_k\vert^3 \vert \sqrt{\det g}-\sqrt{\det g_k}\vert\, dx. 
\end{align*} 
Again, since $u_k$ is bounded in $W^{1,q}$ for some $q > 3$, this will go to zero by H\"older's inequality provided $P$ is sufficiently large depending on $\eps$.  The proposition now follows. 
\end{proof}

\bibliographystyle{plain}
\bibliography{bibliography}

\end{document}